\documentclass[12pt,a4paper]{amsart}
\usepackage{algorithm}
\usepackage{algorithmic}
\usepackage{amssymb}
\usepackage{eucal}
\usepackage{graphicx}
\usepackage{amsmath}
\usepackage{amscd}
\usepackage[all]{xy}           
\usepackage{tikz}
\usepackage{tikz-qtree}
\usepackage{amsfonts,latexsym}
\usepackage{xspace}
\usepackage[T1]{fontenc}
\usepackage{tikz-cd}
\usepackage{txfonts}

\usepackage{epsfig}
\usepackage{float}
\usepackage{color}
\usepackage{fancybox}
\usepackage{forest} 
\usepackage{colordvi}
\usepackage{multicol}
\usepackage{colordvi}
\ifpdf
\usepackage[colorlinks,final,backref=page,hyperindex]{hyperref}
\else
\usepackage[colorlinks,final,backref=page,hyperindex]{hyperref}
\fi
\usepackage[active]{srcltx} 
\usepackage{mathrsfs} 
\usepackage{diagbox}

\newcommand{\nc}{\newcommand}
\newcommand{\delete}[1]{}

\nc{\mlabel}[1]{\label{#1}}  
\nc{\mcite}[1]{\cite{#1}}  
\nc{\mref}[1]{\ref{#1}}  
\nc{\meqref}[1]{~\eqref{#1}} 
\nc{\mbibitem}[1]{\bibitem{#1}} 

\delete{
\nc{\mlabel}[1]{\label{#1}  
{\hfill \hspace{1cm}{\bf{{\ }\hfill(#1)}}}}
\nc{\mcite}[1]{\cite{#1}{{\bf{{\ }(#1)}}}}  
\nc{\mref}[1]{\ref{#1}{{\bf{{\ }(#1)}}}}  
\nc{\meqref}[1]{~\eqref{#1}{{\bf{{\ }(#1)}}}} 
\nc{\mbibitem}[1]{\bibitem[\bf #1]{#1}} 
}

\newtheorem{thm}{Theorem}[section]
\newtheorem{lem}[thm]{Lemma}

\newtheorem{prop}[thm]{Proposition}

\newtheorem{lm}[thm]{Lemma}

\theoremstyle{definition}
\newtheorem{defi}[thm]{Definition}
\newtheorem{ex}[thm]{Example}

\nc{\tred}[1]{\textcolor{red}{#1}}
\nc{\tblue}[1]{\textcolor{blue}{#1}}
\nc{\tgreen}[1]{\textcolor{green}{#1}}
\nc{\tpurple}[1]{\textcolor{purple}{#1}}
\nc{\btred}[1]{\textcolor{red}{\bf #1}}
\nc{\btblue}[1]{\textcolor{blue}{\bf #1}}
\nc{\btgreen}[1]{\textcolor{green}{\bf #1}}
\nc{\btpurple}[1]{\textcolor{purple}{\bf #1}}

\makeatletter

\newcommand*\bigcdot{\mathpalette\bigcdot@{.5}}
\newcommand*\bigcdot@[2]{\mathbin{\vcenter{\hbox{\scalebox{#2}{$\m@th#1\bullet$}}}}}
\makeatother

\newtheorem{innercustomgeneric}{\customgenericname}
\providecommand{\customgenericname}{}
\newcommand{\newcustomtheorem}[2]{%
\newenvironment{#1}[1]
{%
\renewcommand\customgenericname{#2}%
\renewcommand\theinnercustomgeneric{##1}%
\innercustomgeneric
}
{\endinnercustomgeneric}
}
\newcustomtheorem{customthm}{Theorem}

\theoremstyle{definition}

\newtheorem{remark}[thm]{Remark}

\nc{\name}[1]{{\bf #1}}

\nc{\ot}{\otimes}
\nc{\bfk}{\mathbf{k}}
\nc{\id}{\mathrm{id}}
\nc{\calf}{\mathcal{F}}
\nc{\calp}{\mathcal{P}}
\nc{\calq}{\mathcal{Q}}

\allowdisplaybreaks

\newcommand{\N}{\mathbb{N}}
\newcommand{\Z}{\mathbb{Z}}

\nc{\Res}{\mathrm{Res}}

\def \ra {\rightarrow}

\def\<{\langle}
\def\>{\rangle}

\nc{\vspa}{\vspace{-.1cm}}
\nc{\vspb}{\vspace{-.2cm}}
\nc{\vspc}{\vspace{-.3cm}}
\nc{\vspd}{\vspace{-.4cm}}
\nc{\vspe}{\vspace{-.5cm}}

\nc{\mim}{\mathrm{im}}

\nc{\supp}{\text{supp}}
\nc{\Map}{\text{Map}}
\nc{\CSG}{\text{CSG}}

\nc{\fncmna}{$Nov_A^{\rm NC}[X]$\xspace}	

\nc{\CM}{\text{CM}}		
\nc{\KCM}{\text{KCM}}	%

\nc{\mda}{multi-differential algebra}	
\nc{\mdas}{multi-differential algebras}	
\nc{\Mda}{Multi-differential algebra}	
\nc{\Mdas}{Multi-differential algebras}	

\nc{\cmdca}{commuting multi-differential commutative algebra\xspace}
\nc{\cmdnca}{commuting multi-differential noncommutative algebra\xspace}
\nc{\ncmdca}{noncommuting multi-differential commutative algebra\xspace}
\nc{\ncmdnca}{noncommuting multi-differential noncommutative algebra\xspace}

\nc{\cmdcas}{commuting multi-differential commutative algebras\xspace}
\nc{\cmdncas}{commuting multi-differential noncommutative algebras\xspace}
\nc{\ncmdcas}{noncommuting multi-differential commutative algebras\xspace}
\nc{\ncmdncas}{noncommuting multi-differential noncommutative algebras\xspace}

\nc{\Cmdca}{Commuting multi-differential commutative algebra\xspace}
\nc{\Cmdnca}{Commuting multi-differential noncommutative algebra\xspace}
\nc{\Ncmdca}{Noncommuting multi-differential commutative algebra\xspace}
\nc{\Ncmdnca}{Noncommuting multi-differential noncommutative algebra\xspace}

\nc{\Cmdcas}{Commuting multi-differential commutative algebras\xspace}
\nc{\Cmdncas}{Commuting multi-differential noncommutative algebras\xspace}
\nc{\Ncmdcas}{Noncommuting multi-differential commutative algebras\xspace}
\nc{\Ncmdncas}{Noncommuting multi-differential noncommutative algebras\xspace}

\nc{\cmdcac}{\rm{CMDCA}\xspace}	
\nc{\cmdncac}{\rm{CMDNCA}\xspace}	
\nc{\ncmdcac}{\rm{NCMDCA}\xspace}	
\nc{\ncmdncac}{\rm{NCMDNCA}\xspace} 

\nc\MNE{\mathrm{MNE}_{\Omega}(\Omegax)}
\nc\fmna{\mathrm{Nov}_{\Omega}(\Omegax)}
\nc\fmma{\mathrm{Mag}_{\Omega}(\Omegax)}
\nc\fmdc{\mathrm{CD}_{\Omega}(\Omegax)}
\nc\fmdca{\mathrm{PCD}_{\Omega}(\Omegax)}

\nc\Omegax{X}
\nc\NE{\text{NE}(\Omegax)}
\nc\fnp{\triangleright} 
\nc\fna{\mathrm{Nov}(\Omegax)}
\nc\fma{\mathrm{Mag}(\Omegax)}
\nc\fdc{\mathrm{CD}(\Omegax)}
\nc\fdca{\mathrm{CD}(\Omegax)_0}
\nc{\PBTN}{\mathrm{PBT}(\Omegax)}
\nc{\PMTN}{\mathrm{PMT}(\Omegax)}
\nc{\BTN}{\mathrm{BT}(\Omegax)}
\nc{\MTN}{\mathrm{MT}(\Omegax)}

\nc{\Mag}{\text{Mag}}
\nc{\Nov}{\text{Nov}}
\nc{\I}{\text{I}}
\nc{\PBT}{\text{PBT}}
\nc{\PMT}{\text{PMT}}
\nc{\BT}{\text{BT}}
\nc{\MT}{\text{MT}}

\nc{\cmna}{multi-Novikov algebra\xspace}
\nc{\cmncna}{multi-noncommutative Novikov algebra\xspace}
\nc{\ncmna}{noncommuting multi-Novikov algebra\xspace}
\nc{\ncmncna}{noncommuting multi-noncommutative Novikov algebra\xspace}

\nc{\cmdas}{multi-differential algebras\xspace}
\nc{\cmnas}{multi-Novikov algebras\xspace}
\nc{\cmncnas}{multi-noncommutative Novikov algebras\xspace}
\nc{\ncmnas}{noncommuting multi-Novikov algebras\xspace}
\nc{\ncmncnas}{noncommuting multi-noncommutative Novikov algebras\xspace}

\nc{\Cmna}{Multi-Novikov algebra\xspace}
\nc{\Cmncna}{Multi-noncommutative Novikov algebra\xspace}
\nc{\Ncmna}{Noncommuting multi-Novikov algebra\xspace}
\nc{\Ncmncna}{Noncommuting multi-noncommutative Novikov algebra\xspace}

\nc{\Cmnas}{Multi-Novikov algebras\xspace}
\nc{\Cmncnas}{Multi-noncommutative Novikov algebras\xspace}
\nc{\Ncmnas}{Noncommuting multi-Novikov algebras\xspace}
\nc{\Ncmncnas}{Noncommuting multi-noncommutative Novikov algebras\xspace}

\nc{\ad}{\mathrm{ad}}

\nc{\freecmdiff}{F_{\mathrm{CDiff}}^{C}(X)}	

\nc{\cmdiffpoly}{\bfk_{\Omega}^{\mathrm{C}}\{X\}}	
\nc{\mapfinsupp}{\mathrm{Map}^{\mathrm{fs}}}
\nc{\NC}{\mathrm{NC}}
\nc{\ncmder}{\Delta_{\NC,\Omega}}
\nc\NMNE{ \text{\rm MNE}_{\text{\rm NC},\Omega}(\Omegax)}
\nc\fnmna{ \text{\rm Nov}_{\text{\rm NC},\Omega}(\Omegax)}
\nc\fnmdc{ \bfk_{\text{\rm NC},\Omega}\{\Omegax\}}
\nc\fnmdca{ \text{\rm PCD}_{\text{\rm NC},\Omega}(\Omegax)}

\begin{document}

\title[General multi-Novikov algebras]{General multi-Novikov algebras, multi-differential algebras and their free constructions}

\author{Xiaoyan Wang}
\address{Department of School of Mathematical Sciences, East China Normal University, Shanghai 200241, China}
\email{wangxy@math.ecnu.edu.cn}

\author{Li Guo}
\address{
Department of Mathematics and Computer Science,
Rutgers University,
Newark, NJ 07102, United States}
\email{liguo@rutgers.edu}

\author{Huhu Zhang}
\address{School of Mathematics and Statistics
Yulin University,
Yulin, Shaanxi 719000, China}
\email{huhuzhang@yulinu.edu.cn}

\date{\today}

\begin{abstract}
Motivated by the recent development of noncommutative Novikov algebras and multi-Novikov algebras from the study of regularity structures of stochastic PDEs, this paper gives a general approach to study various multi-Novikov algebras and multi-differential algebras, with close connection with Poisson algebras. The construction of S. Gelfand of Novikov algebras from differential commutative algebras is generalized to this context. Free noncommuting multi-Novikov algebras are constructed from typed decorated rooted trees and from noncommuting multi-differential polynomials with populated conditions. 
\end{abstract}

\subjclass[2020]{
17A30, 
17A50, 
12H05, 
17D25, 
17B63, 
17A36, 
16W25 
}

\keywords{Novikov algebra, differential algebra, multi-Novikov algebra, multi-differential algebra, Poisson algebra, free object, multi-index}

\maketitle

\vspace{-1.3cm}

\tableofcontents

\vspace{-1.5cm}

\allowdisplaybreaks

\section{Introduction}
\mlabel{s:intro}
This paper studies variations of multi-Novikov algebras and their inductions from multi-differential algebras with not necessarily commuting derivations. Free noncommuting multi-Novikov algebras are constructed by typed decorated rooted trees and multi-indices. 

\subsection{Novikov algebras and differential algebras}
\mlabel{ss:diffalg}

\subsubsection{Novikov algebras} 
The structure of Novikov algebras first appeared in the study of Hamiltonian operators in connection with formal calculus of variations~ \mcite{GD1} and then independently discovered in the context of classification of linear Poisson brackets of hydrodynamical type \mcite{BN}. The term Novikov algebras was coined by  Osborn in 1992 \mcite{O1}. 

A \textbf{Novikov algebra} is a nonassociative algebra $(A,\circ)$ where the multiplication $\circ$ satisfies the following {\bf left-symmetric} and {\bf right-commutative} identities:
\vspb
\begin{eqnarray}
\mlabel{eq:n1}
(x\circ y)\circ z-x\circ(y\circ z)&=&(y\circ x)\circ z-y\circ(x\circ z), \\
\mlabel{eq:n2}
(x\circ y)\circ z&=&(x\circ z)\circ y, \quad x, y, z\in A.
\end{eqnarray}
Denoting $(x,y,z):=(x\circ y)\circ z-x\circ (y\circ z)$, then the left-symmetric identity means $(x,y,z)=(y,x,z)$.
If only the left-symmetric property Eq. \meqref{eq:n1} is satisfied, then the algebra is called the {\bf pre-Lie algebra} (or the {\bf left-symmetric algebra}). 

Geometrically, a Novikov algebra corresponds to a left-invariant torsion-free flat connection of the Lie group whose Lie algebra is isomorphic to the commutator Lie algebra of the Novikov algebra\,\mcite{K}. Novikov algebras arose in many areas of mathematics and physics, especially in the recent study of stochastic PDEs and numerical methods\,\mcite{BrD,BEH,BrH}. Lately, Novikov algebras have been studied on the level of operads~\mcite{KSO}, and Hopf algebra of decorated multi-indices~\cite{DIU,Fo2,ZGM}.

Classification and construction problems on Novikov algebras have been studied for several decades. Zelmanov~\mcite{Ze} proved that any finite-dimensional simple Novikov algebra over an algebraically closed field with characteristic $0$ is one dimensional. Hence he answered  Novikov’s question that there are no nontrivial simple finite dimensional Novikov algebras over an algebraically closed field of characteristic 0~\mcite{No}. Osborn classified simple Novikov algebras with an idempotent element and some modules over such algebras \mcite{O2}.  In 1992, he classified infinite dimensional simple Novikov algebras with nilpotent elements over a field of characteristic zero and finite dimensional simple Novikov algebras with nilpotent elements over a field of characteristic $p>0$ \mcite{O3}.  Later, a series of classification were given by Xu, including a complete classification of finite-dimensional simple Novikov algebras and their irreducible modules over an algebraically closed field with prime characteristic and the classification of simple Novikov algebras over an algebraically closed field of characteristic zero \mcite{X1,X3}. He also introduced the Novikov-Poisson algebra in order to study the tensor theory of Novikov algebras and proved that any known simple Novikov algebra can be viewed as the Novikov algebraic structure of a Novikov-Poisson algebra~\mcite{X2}. Bai and Meng carried out a series of research on low dimensional Novikov algebras, such as the structure and classification\,\mcite{BM1,BM2,BM3}. 
In 2011, Burde and Graaf gave a systematic method to classify Novikov algebras with a given associated Lie algebra\,\mcite{BG}.

\subsubsection{Novikov algebras from differential algebras}

A classical construction of a Novikov algebra is given by S. Gelfand~\mcite{GD1}. 
Let $(A,\cdot)$ be a commutative associative algebra equipped with a derivation $D$, that is $(A,\cdot,D)$ is a differential commutative algebra\,\mcite{Ko,PS}. Define
\vspb
\begin{equation}
a\circ b:= a\cdot D(b), \ a, b\in A.
\mlabel{eq:diffprod}
\end{equation}
Then $(A,\circ)$ is a Novikov algebra. 

Gelfand's construction is of fundamental importance in the theory of Novikov algebras. All complex Novikov algebras in dimensions no more than three and many important infinite-dimensional simple Novikov algebras can be realized by this construction or its linear deformations~\mcite{BM,BM1,X3}. It is eventually shown that every Novikov algebra is isomorphic to a subalgebra of some Novikov algebra obtained this way~\mcite{BCZ,DL}. Moreover, Gelfand's construction provides a natural right Novikov algebra structure on the vector space of Laurent polynomials~\mcite{HBG}, which is closely related with Novikov algebra affinization \mcite{BN}. 

\subsubsection{Free Novikov algebras}

Due to the importance of free objects, free Novikov algebras were studied. In 2002, Dzhumadil'daev and L\"ofwall obtained a basis for the free Novikov algebra using trees and the free differential commutative algebra~\mcite{DL}. 
In 2018, Tuniyaz, Bokut, Xiryazidin and Obul have found a magmatic Gr\"obner-Shirshov basis
of a free Novikov algebra over any field (or commutative ring) such that the corresponding irreducible words are the Dzhumadil'daev-Löfwall words~\mcite{TBXO}. 
In 2019, Zhang, Chen and Bokut obtained bases for Novikov superalgebras over a field of characteristic zero~\mcite{ZCB}. Free Novikov algebras have been applied to the study of stochastic PDEs\,\mcite{BrD,ZGM}.

\subsection{New notions related to Novikov algebras and differential algebras}
In addition to all the developments on Novikov algebras, there have been generalizations of Novikov algebras in recent years. 

One of the generalizations was given in 2023. Extending S. Gelfand's construction of Novikov algebras from differential commutative algebras, Sartayev and Kolesnikov introduced the notion of noncommutative Novikov algebras and showed that they can be induced from differential algebras in which the multiplication is not commutative\,\mcite{SK}. Free noncommutative Novikov algebras were constructed~\mcite{DS}. Some recent studies on other structures related to Novikov algebras can be found in~\mcite{AS,DS2,GGHZ,KMS,KN,ZZWG,ZC}

Another generalization was also given in 2023. Motivated by the landmark work of M. Hairer and coauthors on regularity structures in stochastic PDEs~\mcite{BHZ,Ha}, Bruned and Dotsenko introduced (commutative) multi-Novikov algebras, a generalization of Novikov algebras with multiple binary operations indexed by a given set~\mcite{BrD}.  
They showed that the multi-indices introduced in the context of singular stochastic PDEs can be interpreted as free multi-Novikov algebras. 
Here a critical role is played by generalizing the classical construction of S. Gelfand to commutative algebras with multiple commuting derivations. The induced multi-Novikov algebras are used to study stochastic PDEs~\mcite{BrD} following the theory of regularity structures invented by M.~Hairer~\mcite{Ha,BHZ}.	

\subsection{Layout of the paper}
Inspired by these recent developments on Novikov algebras and their generalizations, it is desirable to organize them in a uniform framework. For this purpose, we give a systematic study on the variations of multi-Novikov algebras that can be induced from multi-differential algebras, without assuming that the derivations are commuting,  or that the multiplications are commutative. These variations of multi-Novikov algebras are shown to be induced from these more general multi-differential algebras. Properties and examples of these algebras are given. The commonly studied Poisson algebras are special cases of the new kinds of multi-differential algebras and multi-Novikov algebras. 

Free objects in the corresponding categories are constructed, generalizing the previous constructions of free Novikov algebras and multi-Novikov algebras. Furthermore, by working in the context of operads and as applications of free multi-differential algebras, we show that these newly introduced variations of Novikov algebras, as well as the multi-Novikov algebras, in particular Novikov algebras, are the exact induced structures of multi-differential algebras with various commutativity conditions relaxed. 

The overall layout of the paper is as follows. 

In Section~\mref{s:mda}, we first define various multi-differential algebras, where the derivations might be commuting or noncommuting, and the algebra can be commutative or noncommutative. These conditions lead to four types of multi-differential algebras: \cmdcas, \cmdncas, \ncmdcas, and \ncmdncas. Natural examples from calculus and algebra are provided. Especially, the notion of Poisson algebras can be interpreted in terms of noncommuting multi-differential algebras. The free objects in the corresponding four categories are constructed.

In Section~\mref{s:mnov}, we introduce the notions of \cmncnas, \ncmnas, and \ncmncnas.  They are shown to be induced from the four types of multi-differential algebras studied in Section~\mref{s:mda}, as depicted in the diagram
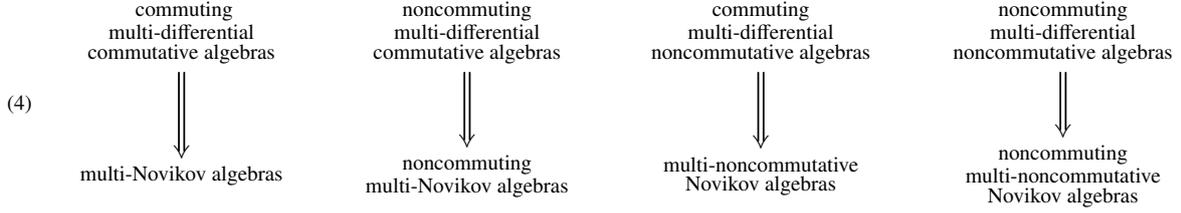
\begin{equation}
\tiny{
\begin{split}
\xymatrix{
\txt{ commuting \\multi-differential\\ commutative algebras} \ar@2{->}[d] &\txt{noncommuting\\ multi-differential\\ commutative algebras} \ar@2{->}[d]&\txt{commuting\\ multi-differential\\ noncommutative algebras} \ar@2{->}[d]&\txt{noncommuting\\ multi-differential\\ noncommutative algebras} \ar@2{->}[d]\\
\txt{multi-Novikov algebras}
&\txt{noncommuting\\ multi-Novikov algebras} &\txt{multi-noncommutative \\Novikov algebras}  
& \txt{noncommuting\\ multi-noncommutative\\ Novikov algebras} 
}
\end{split}}
\mlabel{e:inddiag}
\end{equation}

In Section~\mref{s:fnmna}, we give two explicit constructions of free \ncmnas from free \ncmdcas and from rooted trees, generalizing or expanding the constructions of Dzhumadil'daev-L\"ofwall~\mcite{DL} and Bruned-Dotsenko\,\mcite{BrD}. The main theorem is Theorem\,\mref{t:fncmncasubnov}. 

\smallskip

\noindent
{\bf Notations.} In this paper, we let $\N$ be the set of natural numbers (including $0$), $\N^+$ be the set of positive natural numbers. Let $\bfk$ be a given ground field for all the vector spaces, algebras, linear maps and tensor products. By a (commutative) algebra we mean an associative (commutative) algebra.  

\section{Multi-differential algebras and their free objects}
\mlabel{s:mda}
In this section, we introduce the notions of differential commutative or noncommutative algebras with multiple derivations which are commuting or noncommuting, leading to the notions of \cmdncas, \ncmdcas and \ncmdncas.
We then construct the free objects in the corresponding categories.

\subsection{Notions and examples of \cmdas}
\mlabel{ss:mdas}
We begin with recalling the notion of differential algebras. 

A {\bf derivation} on a not necessarily associative algebra $(R,\cdot)$ is a linear map $d:R\to R$ that satisfies the {\bf Leibniz rule}:
\begin{equation}
d(x\cdot y)=d(x)\cdot y+x\cdot d(y), \quad x, y\in R. 
\mlabel{eq:der}
\end{equation}

The classical notion of a \textbf{differential algebra}~\mcite{Ko,PS,Ri} is an algebra $(R, \cdot)$ equipped with a derivation $d: R\ra R$. Often the algebra is assumed to be a commutative algebra or a field. A differential algebra might also have multiple commuting derivations. 

We now introduce a framework to include several structures with multiple derivations. 

\begin{defi}
Let $\Omega$ be a set. Denote $\partial_\Omega:= \{\partial_\omega~|~\omega \in \Omega\}$.
\begin{enumerate} 
\item \mcite{BrD} A \textbf{\cmdca} indexed by $\Omega$ is a commutative algebra $R$ equipped with derivations  $\partial_\Omega$ which are pairwise commuting. 
\item A \textbf{\cmdnca} indexed by $\Omega$ is a noncommutative algebra $R$ equipped with derivations $\partial_\Omega$ which are pairwise commuting.  	
\item A \textbf{\ncmdca} indexed by $\Omega$ is a commutative algebra $R$ equipped with derivations $\partial_\Omega$ which are not necessarily pairwise commuting.  	
\item A \textbf{\ncmdnca} indexed by $\Omega$ is a noncommutative algebra $R$ equipped with derivations $\partial_\Omega$ which are not necessarily pairwise commuting.  	
\end{enumerate}
\mlabel{d:mdas}
\end{defi}
These structures can be organized in the following table. 
\begin{small}
\begin{table}[h]
\centering
\begin{tabular}{|c|c|c|}
\hline
\diagbox{multiplication}{derivations}&\text{commuting}&\text{noncommuting}\\
\hline
\text{commutative}&\txt{commuting\\ multi-differential\\ commutative algebras\\(CMDCAs)}&\txt{noncommuting\\ multi-differential\\ commutative algebras\\(NCMDCAs)}\\
\hline
\text{noncommutative}&\txt{commuting\\ multi-differential\\ noncommutative algebras\\(CMDNCAs)}&\txt{noncommuting\\ multi-differential\\ noncommutative algebras\\(NCMDNCAs)}\\
\hline
\end{tabular}
\end{table}
\end{small}

\begin{defi}
Let $(R, \partial_\Omega)$ and $(S, \delta_\Omega)$ be \cmdcas. A \textbf{homomorphim} from $(R, \partial_\Omega)$ to $(S, \delta_\Omega)$ is an algebra homomorphism $f: R \rightarrow S$ such that $f\circ \partial_\omega=\delta_\omega \circ f$ for all $\omega\in \Omega$. Let \cmdcac denote the category of \cmdcas.

The same notion can be defined for the other types of multi-differential algebras in Definition \mref{d:mdas}. We let \cmdncac, \ncmdcac and \ncmdncac denote the corresponding categories.
\end{defi}

We give some examples of multi-differential algebras. 

\begin{ex}
\begin{enumerate}
\item Let $\bfk[X]$ be a polynomial and let $\partial_X:=\{\partial_x:=\frac{\partial}{\partial x}\,|\,x\in X\}$	be the set of partial derivatives. Then $(\bfk[X],\partial_X)$ is a commuting multi-differential commutative algebra. Likewise, $(\bfk\langle X\rangle, \partial_X)$ is a commuting multi-differential noncommutative algebra. 
\item 
Let $\bfk[x]$ be the polynomial algebra in variable $x$. As is well-known, the set
$$\mathcal{D}:=\bigg\{f(x)\frac{d}{dx}\,\bigg|\,f(x)\in \bfk[x]\bigg\}$$ 
is a family of derivations on $\bfk[x]$ that might not be pairwise commuting. For a nonempty set $\Omega$ and a set map 
$\partial:\Omega\to \mathcal{D}$, denote $\partial_\Omega:=\Big\{\partial_\omega:=\partial(\omega)\,\Big|\,\omega\in \Omega\Big\}$. Then we obtain a noncommuting multi-differential algebra $(\bfk[x],\partial_\Omega)$. 
\mlabel{i:pab}
\item More generally, for a commutative algebra $A$, let $\text{Der}(A)\subseteq \text{End}(A)$ be the set of derivations on $A$. Then $\text{Der}(A)$ is a Lie algebra with respect to the bilinear product $$[,]:\text{Der}(A)\otimes \text{Der}(A) \ra\text{Der}(A), [d_1,d_2]=d_1\circ d_2-d_2\circ d_1,\quad d_1,d_2\in \text{Der}(A).$$ 
For a given set $\Omega$ and a map $\partial: \Omega \ra \text{Der}(A)$. Define $$\partial_{\Omega}:=\{\partial_\omega:=\partial(\omega)\,|\,\omega\in \Omega\}.$$ 
Then $(A,\partial_{\Omega})$ is a \ncmdca. In fact, every \ncmdca can be realized this way. 
\mlabel{i:pa1}
\end{enumerate}
\mlabel{e:pa}
\end{ex}

We next show that a Poisson algebra can be naturally interpreted as a noncommuting multi-differential algebra. 

Recall that a {\bf Poisson algebra} is a $\bfk$-module equipped with bilinear products $\cdot$ and $[,]$, such that $(P,\cdot)$ is a commutative algebra, $(P,[,])$ is a Lie algebra, and 
\begin{equation}
[x,y\cdot z]=[x,y]\cdot z+y\cdot [x,z], \quad x, y, z\in P.
\mlabel{eq:poisleib}
\end{equation}

\begin{prop}
Let $(A,\cdot)$ be a commutative algebra equipped with a Lie bracket $[\,,\,]$. 
Consider the adjoint action
$$\ad:A \ra \mathrm{End}(A), \quad  \ad(x)(y):=\ad_x(y)=[x,y], \quad x, y\in A.$$
Then the triple $(A,\cdot,[\,,\,])$ is a Poisson algebra if and only if the triple $(A,\cdot,\ad(A))$ is a noncommuting multi-differential algebra. 
\mlabel{p:poisson}
\end{prop}

\begin{proof}
If $(A,\cdot,[,])$ is a Poisson algebra. Then $\ad_A=\{\ad_x\,|\,x\in A\}$ is a family of derivations on the commutative algebra $A$. This means that $(A,\ad(A))$ is a noncommuting multi-differential commutative algebra. 

Conversely, suppose that $(A,\cdot,\ad(A))$ is a noncommuting multi-differential algebra. Then the subset $\ad(A)\subseteq \text{End}(A)$ is in the subspace $\text{Der}(A)$ of $\mathrm{End}(A)$. Therefore, Eq.~\meqref{eq:poisleib} holds and $(A,\cdot,[,]_A)$ is a Poisson algebra. 
\end{proof}

\subsection{Free multi-differential algebras}
\mlabel{ss:fcmdca}	
To fix notations and provide a unified context, we first recall the constructions of free commuting multi-differential commutative and noncommutative algebras. We then give the constructions of free noncommuting multi-differential commutative and noncommutative algebras. 

For a set $Y$, let $M(Y)$ denote the free commutative monoid generated by $Y$, realized as the set of maps with finite supports:
\begin{equation}
M(Y):=\big\{\alpha:Y\to \N\,\big|\, |\supp(\alpha)|<\infty\big\}.
\end{equation}
Here the support of $\alpha$ is $\supp(\alpha):=\{y\in Y\,|\,\alpha(y)\neq 0\}$. 
The multiplication in $M(Y)$ is given by 
$$ (\alpha\cdot \beta)(y):=\alpha(y)+\beta(y), \quad y\in Y.$$
Identifying $\alpha\in M(Y)$ with 
$$\prod_{y\in \supp(\alpha)}y^{\alpha(y)}=y_1^{\alpha(y_1)}\cdots y_k^{\alpha(y_k)}$$ 
when $\supp(\alpha)=\{y_1,\ldots,y_k\}$, we obtain the usual construction of the free commutative monoid: 
\begin{equation}
M(Y):=\{ y_1^{\alpha_1}\cdots y_m^{\alpha_m}\,|\, y_i\in Y, \alpha_i\geq 1,m\geq 1\} \cup \{1\}.
\end{equation}

Now for nonempty sets $\Omega$ and $X$, define the set $\partial_{\Omega}:=\{\partial_\omega\,|\,\omega\in \Omega\}$ and define the set of {\bf variables from commuting multi-derivations} to be 
\begin{equation}
\Delta_{\Omega}(X):=M(\partial_\Omega)\times X=\bigg\{\Big( \prod_{\omega\in \supp(\alpha)}\partial_\omega^{\alpha(\omega)},x\Big)\bigg|\alpha\in M(\partial_\Omega), x\in X\bigg\}.
\mlabel{eq:mdiffvar}
\end{equation}
We also use the functional notation 
$$ \prod_{\omega\in \supp(\alpha)}\partial_\omega^{\alpha(\omega)}(x):=\Big(\prod_{\omega\in\supp(\alpha)}\partial_\omega^{\alpha(\omega)}\Big)(x):=\Big( \prod_{\omega\in \supp(\alpha)}\partial_\omega^{\alpha(\omega)},x\Big).$$

For $\omega\in \Omega$, define 
$$D_{\omega}: \Delta_{\Omega}(X)\ra \Delta_{\Omega}(X),$$ 
$$D_\omega \Big(\prod_{\tau\in\supp(\alpha)}\partial_\tau^{\alpha(\tau)}(x)\Big):=\partial_\omega\prod_{\tau\in\supp(\alpha)}\partial_{\tau}^{\alpha(\tau)}(x)=\partial_{\omega}^{\alpha(\omega)+1}\prod_{\tau\in\supp(\alpha),\tau\neq\omega}\partial_{\tau}^{\alpha(\tau)}(x).$$ 

Take the polynomial algebra 
\begin{equation}
\bfk_\Omega\{X\}:=\bfk[\Delta_\Omega(X)]=\bfk M(M(\partial_{\Omega}) \times X).
\end{equation}
Extend $D_\omega$ to $D_\omega:\bfk_\Omega\{X\} \ra \bfk_\Omega\{X\}$ by the Leibniz rule and the linearity. Let 
$$i:X\ra \bfk_\Omega\{X\}, x\mapsto (1,x),$$ 
be the natural injection. Then a fundamental result on differential algebra is the following construction of free objects.
\begin{thm} \mcite{Ko,PS}
The triple	$(\bfk_\Omega\{X\}, D_\Omega,i)$, where $D_\Omega=\{D_\omega~|~\omega\in \Omega\}$, is the free \cmdca on $X$ with commuting derivation set $D_\Omega$.
\mlabel{t:fcmdca}
\end{thm}

Likewise, take noncommutative polynomial algebra in variables $M(\partial_\Omega)\times X$: 
\begin{equation}
\bfk_{\Omega}^{\rm NC}\{ X\}:=\bfk \langle M(\partial_\Omega)\times X \rangle.
\mlabel{eq:diffncpoly}
\end{equation}
Again extend $D_\omega$ by the Leibniz rule and linearity to $\bfk_{\Omega}^{\rm NC}\{ X\}$. Let $i:X\ra \bfk_\Omega^{\rm NC}\{ X\}$ be the injection.
\begin{thm} \mcite{GK}
The triple	$(\bfk_{\Omega}^{\rm NC}\{ X\}, D_\Omega,i)$, where $D_\Omega=\{D_\omega~|~\omega\in \Omega\}$, is the free \cmdnca on $X$ with commuting derivation set $D_\Omega$.
\mlabel{t:fcmdnca}
\end{thm}

\begin{remark} We will use the convention that a superscript NC indicates noncommuative variables and a subscript NC indicates noncommuting derivations.
\mlabel{r:nc}
\end{remark}

We now consider the case when the derivations are not necesarily commuting.

First for a nonempty set $Y$, let $M_{\NC}(Y)$ be the free monoid on $Y$ realized as words:
\begin{equation}
M_{\NC}(Y):=\{y_1\cdots y_k\,|\, y_i\in Y, 1\leq i\leq k, k\geq 1\}\cup \{1\}
\mlabel{eq:ncmonoid}
\end{equation}
with the concatenation multiplication. 

Next let $\Omega$ and $X$ be nonempty sets. Let 
\begin{equation}
\partial_\Omega:=\{\partial_\omega\,|\, \omega\in\Omega\}
\mlabel{eq:domega}
\end{equation}
be a set of symbols parameterized by $\Omega$ that will play the role of multi-derivations on the free \ncmdca. 
Define the set of {\bf variables from noncommuting multi-derivations} to be 
\begin{equation}
\ncmder(X):=M_{\NC}(\partial_\Omega)\times X.
\mlabel{eq:ncmdiffvar}
\end{equation}
For $\omega\in \Omega$, define 
\begin{equation}
D_{\omega}: \ncmder(X) \ra \ncmder(X), \quad D_\omega \big((\partial_{\omega_1}\cdots \partial_{\omega_k})(x)\big):=(\partial_\omega \partial_{\omega_1}\cdots \partial_{\omega_k})(x).
\mlabel{eq:deronmulti}
\end{equation}

Take the commutative polynomial algebra 
$$\bfk_{{\rm NC},\Omega}\{X\}:=\bfk [\ncmder(X)]:=\bfk M(\ncmder(X))$$
like in Theorem~\mref{t:fcmdca}. Extend $D_\omega$ for $\omega\in \Omega$ by the Leibniz rule to a linear operator on $\bfk_{{\rm NC},\Omega}\{X\}$. Let $i: X\ra \bfk_{{\rm NC},\Omega}\{X\}, x\mapsto (1,x)$ be the natural injection.

\begin{thm}
The pair $(\bfk_{\NC,\Omega}\{X\}, D_\Omega)$ with $i: X\ra \bfk_{{\rm NC},\Omega}\{X\}$ is the free \ncmdca on $X$ with its noncommuting derivation set $D_\Omega$.
\mlabel{t:fncmdca}
\end{thm}
\begin{proof}
Let a \ncmdca $(A, \delta_{\Omega})$ and a set map $f:X\to A$ be given. 
Define $\bar{f}: \bfk_{{\rm NC},\Omega}\{X\}\to A$ by 
$$\bar{f}\Big((\partial_{\omega_{1,1}}\cdots \partial_{\omega_{1,k_1}})(x_1)\cdots (\partial_{\omega_{n,1}}\cdots \partial_{\omega_{n,k_n}})(x_n)\Big):=(\delta_{\omega_{1,1}}\cdots \delta_{\omega_{1,k_1}})\big(f(x_1)\big)\cdots (\delta_{\omega_{n,1}}\cdots \delta_{\omega_{n,k_n}})\big(f(x_n)\big).$$
Then $\bar{f}$ is an algebra homomorphism. Also, for $x\in X$, we have $\bar{f}(i(x))=\bar{f}(x)=f(x)$. Hence $\bar{f}\circ i=f$. 

We next verify $\bar{f}\circ D_\tau=\delta_\tau\circ f$ for $\tau\in \Omega$. In fact, 
\begin{align*}
&(\bar{f}\circ D_\tau)((\partial_{\omega_{1,1}}\cdots \partial_{\omega_{1,k_1}})(x_1)\cdots (\partial_{\omega_{n,1}}\cdots \partial_{\omega_{n,k_n}})(x_n))\\
=&~\bar{f}\Big(\sum_{i=1}^{n}((\partial_{\omega_{1,1}}\cdots \partial_{\omega_{1,k_1}})(x_1)\cdots(\partial_\tau \partial_{\omega_{i,1}}\cdots \partial_{\omega_{i,k_i}}(x_i))\cdots (\partial_{\omega_{n,1}}\cdots \partial_{\omega_{n,k_n}})(x_n)\Big)\\
=&~\sum_{i=1}^{n}((\delta_{\omega_{1,1}}\cdots \delta_{\omega_{1,k_1}})(f(x_1))\cdots(\delta_\tau \delta_{\omega_{i,1}}\cdots \delta_{\omega_{i,k_i}}(f(x_i)))\cdots (\delta_{\omega_{n,1}}\cdots \delta_{\omega_{n,k_n}})(f(x_n))\\
=&~\delta_{\tau}((\delta_{\omega_{1,1}}\cdots \delta_{\omega_{1,k_1}})(f(x_1))\cdots (\delta_{\omega_{n,1}}\cdots \delta_{\omega_{n,k_n}})(f(x_n)))\\
=&~(\delta_{\tau}\circ \bar{f})((\partial_{\omega_{1,1}}\cdots \partial_{\omega_{1,k_1}})(x_1)\cdots (\partial_{\omega_{n,1}}\cdots \partial_{\omega_{n,k_n}})(x_n)).
\end{align*}
Therefore, $\bar{f}$ is a homomorphism of \ncmdcas.

Finally, by construction, the only way to get a \ncmdca homomorphism from $f$ is $\bar{f}$, which guarantees the uniqueness.
\end{proof}

As in the previous case, take the noncommutative polynomial algebra $$\bfk^{{\rm NC}}_{{\rm NC},\Omega}\{X\}:=\bfk\langle \ncmder(X)\rangle:=\bfk M_{{\rm NC}}(\ncmder(X))$$ 
as in Theorem~\mref{t:fcmdnca}. Extend $D_\omega$, $\omega\in \Omega$, to $\bfk^{{\rm NC}}_{{\rm NC},\Omega}\{X\}$ by the Leibniz rule. 
\begin{thm}
The pair $(\bfk^{{\rm NC}}_{{\rm NC},\Omega}\{X\}, D_\Omega)$ with the natural injection $i:X\to \bfk^{\NC}_{\NC,\Omega}\{X\}$ is the free \ncmdnca on $X$ with noncommuting derivation set $D_\Omega$.
\mlabel{t:fncmdnca}
\end{thm}

\begin{proof}
The proof is similar to the one for Theorem \mref{t:fncmdca}.
\end{proof}

\section{Variations of \cmnas and their derivations from variations of \cmdas}
\mlabel{s:mnov}
In this section, we introduce variations of \cmnas. We then show that variations of multi-differential algebras give rise to variations of \cmnas as shown in Diagram~\meqref{e:inddiag}.

\begin{defi}
\mlabel{d:mna}
\begin{enumerate}
\item \mlabel{cmna}
\mcite{BrD} A \textbf{\cmna} is a vector space $N$ equipped with bilinear products $\rhd_\omega$ indexed by a set $\Omega$, which satisfy the following identities.
\begin{eqnarray}
&(x\rhd_\omega y)\rhd_\tau z-x\rhd_\omega (y\rhd_\tau z)=(y\rhd_\omega x)\rhd_\tau z-y\rhd_\omega (x\rhd_\tau z),&
\mlabel{eq:nv1}
\\
&(x\rhd_\omega y)\rhd_\tau z-x\rhd_\omega (y\rhd_\tau z)=(x\rhd_\tau y)\rhd_\omega z-x\rhd_\tau(y\rhd_\omega z),&
\mlabel{eq:nv2}\\
&(x\rhd_\omega y)\rhd_\tau z=(x\rhd_\tau z)\rhd_\omega y, \quad x, y, z\in N, \omega, \tau \in \Omega.&
\mlabel{eq:nv3}
\end{eqnarray} 
\item 
\mlabel{d:ncmna}
A {\bf \ncmna} is a vector space $N$ equipped with bilinear products $\rhd_\omega$ indexed by a set $\Omega$, which satisfy the following identities. 
\begin{eqnarray}
&(x \rhd_\omega y)\rhd_\tau z-x\rhd_\omega (y\rhd_\tau z)=(y\rhd_\omega x)\rhd_\tau z-y\rhd_\omega (x\rhd_\tau z),
\mlabel{eq:ncmn1}
&\\ 
&(x\rhd_\omega y)\rhd_\tau z=(x\rhd_\tau z)\rhd_\omega y,\quad  x,y,z\in N, \omega,\tau\in \Omega.
\mlabel{eq:ncmn2} &
\end{eqnarray}
\item 
\mlabel{d:mnna}
A {\bf \cmncna} is a vector space $N$ equipped with two families of bilinear products 
$$\rhd_\Omega:=\{\rhd_\omega\,|\,\omega\in \Omega\}, \quad 
\lhd_\Omega:=\{\lhd_\omega\,|\,\omega\in \Omega\},$$ 
which satisfy the following identities. 
\begin{eqnarray}
&(x\lhd_\omega y)\rhd_\tau z=x\lhd_\omega(y\rhd_\tau z),
&\mlabel{eq: mnca1}
\\
&(x\rhd_\omega y)\rhd_\tau z-x\lhd_\tau(y\lhd_\omega z)=x\rhd_\tau(y\lhd_\omega z)-(x\rhd_\omega y)\lhd_\tau z,
&\mlabel{eq: mnca2}
\\
&x\rhd_\omega(y\rhd_\tau z)-x\rhd_\tau(y\rhd_\omega z)=x\rhd_\tau(y\lhd_\omega z)-x\rhd_\omega(y\lhd_\tau z),
&
\mlabel{eq: mnca3}
\\
&(x\rhd_\tau y)\rhd_\omega z-x\lhd_\tau(y\lhd_\omega z)=x\rhd_\omega(y\lhd_\tau z)-(x\rhd_\omega y)\lhd_\tau z,
&\mlabel{eq: mnca4}
\\
&(x\lhd_\omega y)\lhd_\tau z-(x\lhd_\tau y)\lhd_\omega z=(x\rhd_\tau y)\lhd_\omega z-(x\rhd_\omega y)\lhd_\tau z,
&\mlabel{eq: mnca5}
\\
&x\lhd_\omega (y\lhd_\tau z)-x\lhd_\tau(y\lhd_\omega z)=(x\rhd_\tau y)\lhd_\omega z-(x\rhd_\omega y)\lhd_\tau z,\ \ x, y, z\in N, \omega, \tau\in \Omega.&
\mlabel{eq: mnca6}
\end{eqnarray}
\item 
\mlabel{d:ncmncna}
A {\bf \ncmncna} is a vector space $N$ equipped with two families of bilinear products 
$$\rhd_\Omega:=\{\rhd_\omega\,|\,\omega\in \Omega\}, \quad 
\lhd_\Omega:=\{\lhd_\omega\,|\,\omega\in \Omega\},$$ 
which satisfy the following identities. 
\begin{eqnarray}
&(x\lhd_\omega y)\rhd_\tau z=x\lhd_\omega(y\rhd_\tau z),&
\mlabel{eq: ncncna1}
\\
&(x\rhd_\omega y)\rhd_\tau z-x\lhd_\tau(y\lhd_\omega z)=x\rhd_\tau(y\lhd_\omega z)-(x\rhd_\omega y)\lhd_\tau z, \ x, y, z\in N, \omega, \tau\in \Omega.&
\mlabel{eq: ncncna2}
\end{eqnarray}
\end{enumerate}
\end{defi}

As in the case of Novikov algebras, the various Novikov algebras are induced by various multi-differential algebras as shown below. We include the case of multi-Novikov algebras for completeness. General studies of induces structures can be found in\,\mcite{GWZ,ZHG}

\begin{thm} \mlabel{p:mdaind}
Let $\Omega$ be a nonempty set. 
\begin{enumerate}
\item 
\mcite{BrD} Let $(A,\partial_\Omega)$ be a commuting multi-differential commutative algebra. Define binary operations $\rhd_\Omega:=\{\rhd_\omega, \omega\in\Omega\}$ on $A$ by 
$$ x \rhd_\omega y:=x \partial_\omega(y), \quad x, y\in A, \omega\in \Omega.$$
Then $(A,\rhd_\Omega)$ is a multi-Novikov algebra. 
\mlabel{p:mnov1}
\item
Let $(A, \partial_\Omega)$ be a \ncmdca. 
Define binary operations $\rhd_\Omega:=\{\rhd_\omega\,|\,\omega\in \Omega\}$ on $A$ by 
$$x\rhd_\omega y=x\partial_\omega (y), \quad x, y\in A, \omega\in \Omega.$$
Then $(A, \rhd_\Omega)$ is a \ncmna. 
\mlabel{p:mnovfromncmdca}
\item 
Let $(A, \partial_\Omega)$ be a \cmdnca. 
Define binary operations 
$$\rhd_\Omega:=\{\rhd_\omega\,|\,\omega\in\Omega\}, \quad  \lhd_\Omega:=\{\lhd_\omega\,|\, \omega\in \Omega\}$$ 
on $A$ by 
\begin{equation} \mlabel{eq:cmnca}
x\rhd_\omega y=x\partial_\omega (y), \quad x\lhd_\omega y=\partial_\omega(x) y\quad x, y\in A, \omega\in \Omega.
\end{equation}
Then $(A,\rhd_{\Omega},\lhd_\Omega)$ is a \cmncna.
\mlabel{p:mnovfromcmdnca}
\item
Let $(A, \partial_\Omega)$ be a \ncmdnca. 
Define binary operations 
$$\rhd_\Omega:=\{\rhd_\omega\,|\,\omega\in\Omega\}, \quad  \lhd_\Omega:=\{\lhd_\omega\,|\, \omega\in \Omega\}$$ on $A$ by 
\begin{equation} \mlabel{eq:ncmnca}
x\rhd_\omega y:=x\partial_\omega (y), \quad x\lhd_\omega y:=\partial_\omega(x) y\quad x, y\in A, \omega\in \Omega.
\end{equation}
Then $(A, \rhd_\Omega,\lhd_\Omega)$ is a \ncmncna.
\mlabel{p:mnovfromncmdnca}
\end{enumerate}
\end{thm}

\begin{proof}
\mref{p:mnov1} This is obtained in \mcite{BrD}. 

\smallskip \noindent
\mref{p:mnovfromncmdca}
We just need to show that the identities (\mref{eq:ncmn1}) and (\mref{eq:ncmn2}) in Definition \mref{d:mna}.\mref{d:ncmna} hold. Indeed, for any $x, y, z\in A$ and $\omega, \tau\in \Omega$, we have 
\begin{align*}
(x \rhd_\omega y)\rhd_\tau z-x\rhd_\omega (y\rhd_\tau z )
&=x\partial_\omega (y)\partial_\tau (z)-x\partial_\omega(y)\partial_\tau (z)-xy(\partial_\omega \partial_\tau z)\\
&=-xy(\partial_\omega\partial_\tau z)\\
&=y\partial_\omega (x)\partial_\tau (z)-y\partial_\omega(x)\partial_\tau (z)-yx(\partial_\omega \partial_\tau)(z)\\
&=(y\rhd_\omega x)\rhd_\tau z-y\rhd_\omega (x\rhd_\tau z),
\end{align*}
and
$$		(x\rhd_\omega y)\rhd_\tau z =x\partial_\omega (y)\partial_\tau (z)
=(x\rhd_\tau z)\rhd_\omega y.  
$$

\smallskip
\noindent
\mref{p:mnovfromcmdnca}
We verify identities (\mref{eq: mnca1}) -- (\mref{eq: mnca6}) in Definition \mref{d:mna}.\mref{d:mnna} as follows. 
$$(x\lhd_\omega y)\rhd_\tau z=\partial_\omega(x)y\partial_\tau(z)
=x\lhd_\omega(y\rhd_\tau z).$$
\begin{align*}
(x\rhd_\omega y)\rhd_\tau z-x\lhd_\tau(y\lhd_\omega z)&=
x\partial_\omega(y)\partial_\tau(z)-\partial_\tau(x)\partial_\omega(y)z\\&=x\partial_\tau\partial_\omega(y)z+x\partial_\omega(y)\partial_\tau(z)-\partial_\tau(x)\partial_\omega(y)z-x\partial_\tau\partial_\omega(y)z\\
&=x\rhd_\tau(y\lhd_\omega z)-(x\rhd_\omega y)\lhd_\tau z.
\end{align*}
\begin{align*}
x\rhd_\omega(y\rhd_\tau z)-x\rhd_\tau(y\rhd_\omega z)&=x\partial_\omega(y)\partial_\tau(z)+xy\partial_\omega\partial_\tau(z)-x\partial_\tau(y)\partial_\omega(z)-xy\partial_\tau\partial_\omega(z)\\
&=x\partial_\tau\partial_\omega(y)z+x\partial_\omega(y)\partial_\tau(z)-x\partial_\omega\partial_\tau(y)z-x\partial_\tau(y)\partial_\omega(z)\\
&=x\rhd_\tau(y\lhd_\omega z)-x\rhd_\omega(y\lhd_\tau z).
\end{align*}
\begin{align*}
(x\rhd_\tau y)\rhd_\omega z-x\lhd_\tau(y\lhd_\omega z)&=x\partial_\tau(y)\partial_\omega(z)-\partial_\tau(x)\partial_\omega(y)z\\
&=x\partial_\omega\partial_\tau(y)z+x\partial_\tau(y)\partial_\omega(z)-\partial_\tau(x)\partial_\omega(y)z-x\partial_\tau\partial_\omega(y)z\\
&=x\rhd_\omega(y\lhd_\tau z)-(x\rhd_\omega y)\lhd_\tau z.
\end{align*}
\begin{align*}
(x\lhd_\omega y)\lhd_\tau z-(x\lhd_\tau y)\lhd_\omega z&=\partial_\tau\partial_\omega(x)yz+\partial_\omega(x)\partial_\tau(y)z-\partial_\omega\partial_\tau(x)yz-\partial_\tau(x)\partial_\omega(y)z\\
&=x\partial_\omega\partial_\tau(y)z+\partial_\omega(x)\partial_\tau(y)z-x\partial_\tau\partial_\omega(y)z-\partial_\tau(x)\partial_\omega(y)z\\
&=(x\rhd_\tau y)\lhd_\omega z-(x\rhd_\omega y)\lhd_\tau z.
\end{align*}
\begin{align*}
x\lhd_\omega (y\lhd_\tau z)-x\lhd_\tau(y\lhd_\omega z)
&=\partial_\omega(x)\partial_\tau(y)z-\partial_\tau(x)\partial_\omega(y)z\\
&=x\partial_\omega\partial_\tau(y)z+\partial_\omega(x)\partial_\tau(y)z-x\partial_\tau\partial_\omega(y)z-\partial_\tau(x)\partial_\omega(y)z\\
&=(x\rhd_\tau y)\lhd_\omega z-(x\rhd_\omega y)\lhd_\tau z.
\end{align*}

\smallskip
\noindent
\mref{p:mnovfromncmdnca}
We verify identities (\mref{eq: ncncna1}) and (\mref{eq: ncncna2}) in Definition \mref{d:mna}.\mref{d:ncmncna} as follows. 
$$(x\lhd_\omega y)\rhd_\tau z=\partial_\omega(x)y\partial_\tau(z)
=x\lhd_\omega(y\rhd_\tau z).$$
\begin{align*}
(x\rhd_\omega y)\rhd_\tau z-x\lhd_\tau(y\lhd_\omega z)&=
x\partial_\omega(y)\partial_\tau(z)-\partial_\tau(x)\partial_\omega(y)z\\&=x\partial_\tau\partial_\omega(y)z+x\partial_\omega(y)\partial_\tau(z)-\partial_\tau(x)\partial_\omega(y)z-x\partial_\tau\partial_\omega(y)z\\&=x\rhd_\tau(y\lhd_\omega z)-(x\rhd_\omega y)\lhd_\tau z. 
\hspace{4cm}	\qedhere
\end{align*}
\end{proof}

\begin{ex}
Let $(\bfk[x],\cdot)$ be the polynomial algebra in variable $x$. Let $$\left\{\left .f(x)\frac{d}{dx}\,\right|\,f(x)\in \bfk[x]\right\}$$ 
be the family of derivations on $\bfk[x]$ as in Example~\mref{e:pa}.\mref{i:pab}. Define a family of binary operations  
$\rhd_{\bfk[x]}=\{\rhd_{f(x)}|f(x)\in \bfk[x]\}$ on $\bfk[x]$ by taking 
$$g(x)\rhd_{f(x)} h(x):=g(x)f(x)\frac{d}{dx}(h(x)), \quad g(x),h(x)\in \bfk[x].$$ 
Then $(\bfk[x],\rhd_{\bfk[x]})$ is a \ncmna.
\end{ex}

\begin{ex}
Let $(A,\cdot,[,]_A)$ be a Poisson algebra. Then we have a \ncmdca $(A,\ad(A))$ as in Proposition \mref{p:poisson}. Define a family of binary operations $\rhd_{\ad(A)}=\{\rhd_\partial|\partial=\ad_z\in \ad(A), z\in A\}$ on $A$, by 
$$x\rhd_\partial y:=x\partial (y)=x\ad_z (y)=x[z,y].$$ 
Then $(A,\rhd_{\ad(A)})$ is a \ncmna. This realizes every Poisson algebra as a \ncmna. 
\end{ex}

\begin{ex}
Let $M_{2\times 2}$ be the linear space of $2\times 2$ matrices with the matrix multiplication. Take a family of pairwisely commutative matrices $\Omega=\{x_n|x_n\in M_{2\times 2}\}$. Define a family of  derivations $\ad_\Omega:=\{\ad_{x_i}|x_i \in \Omega\}$ on $M_{2\times 2}$ by $$\ad_{x_i}(y)=[x_i,y]=x_iy-yx_i, x_i\in \Omega, y\in M_{2\times 2}.$$ 
Then we get a family of pairwise commuting derivations, since for any $2\times 2$ matrix $z$, we have 
\begin{eqnarray*}
\ad_{x_1}\ad_{x_2}(z)&=&x_1(x_2z)-x_1(zx_2)-(x_2z)x_1+(zx_2)x_1\\
&=&x_2(x_1z)-x_2(zx_1)-(x_1z)x_2+(zx_1)x_2\\
&=&\ad_{x_2}\ad_{x_1}(z).
\end{eqnarray*} 
Define a family of binary operations $\rhd_\Omega:=\{\rhd_{x_i}|x_i\in \Omega\}$ on $M_{2\times 2}$ by 
$$y\rhd_{x_i} z:=y\,\ad_{x_i}(z).$$ 
Then $(M_{2\times 2},\rhd_\Omega)$ is a \cmncna.
\end{ex}

\section{Constructions of free noncommuting multi-Novikov algebras}
\mlabel{s:fnmna}
\nc\novpromag{\blacktriangleright}

The purpose of this section is to construct the free objects in the category of noncommuting multi-Novikov algebras defined in Definition~\mref{d:mna}.\mref{d:ncmna}. In fact, we will give two explicit constructions of these free objects.  
We first introduce enough notations to state the main result: Theorem~\mref{t:fncmncasubnov}. The proof will comprise the rest of the paper, adapting the notations and construction of free Novikov algebras and multi-Novikov algebras\,\mcite{BrD,DL}.  

\subsection{Free noncommuting multi-Novikov algebras from typed decorated rooted trees}

Recall that an {\bf ($\Omega$-)multi-magmatic algebra} is a vector space $\text{M}$ equipped with multiplications $\rhd_\omega$ indexed by $\omega$ in a set $\Omega$. 
A standard construction of the free $(\Omega)$-multi-magmatic algebra on $X$ is the vector space 
$$\text{Mag}_\Omega(X):=\bfk \mathcal{T}_\Omega(X)$$ 
with the basis $\mathcal{T}_\Omega(X)$ of planar binary trees whose internal vertices are decorated by $\Omega$ and whose the leafs are decorated by $X$. The binary operation $\novpromag_\omega$ 
on $\text{Mag}_\Omega(X)$ is the grafting of two trees with the new root decorated by $\omega$\,\mcite{BrD}. 

Let $I$ be the two-sided ideal of $\fmma$ generated by all elements of the forms
\begin{align}
(x \novpromag_\omega y)\novpromag_\tau z&-x\novpromag_\omega (y\novpromag_\tau z)-(y\novpromag_\omega x)\novpromag_\tau z+y\novpromag_\omega (x\novpromag_\tau z), \mlabel{e:ncmnrel1}\\
(x\novpromag_\omega y)\novpromag_\tau z&-(x\novpromag_\tau z)\novpromag_\omega y,\quad x,y,z\in \fmma,\omega,\tau\in \Omega.
\mlabel{e:ncmnrel2}
\end{align}
Then by construction, 
$\fmma / I$
is the free noncommuting multi-Novikov algebra on $X$. 
Let $\pi:\fmma\to \fmma / \text{I}$ denote the natural quotient map, and let 
$$i: X\hookrightarrow \fmma\xrightarrow{\pi} \fmma / I.$$ 

Let $u$ be an element in the basis $\mathcal{T}_\Omega(X)$. Then either $u=z\in X$ or $u=y_1\novpromag_{\omega_1} z_1$ is the grafting of the left branch $y_1$ and right branch $z_1$ with root $\omega_1$. In turn, unless $z_1$ is in $X$, it also has a grafting $z_1=y_2 \novpromag_{\omega_2} z_2$, yielding $u=y_1\novpromag_{\omega_1}(y_2\novpromag_{\omega_2} z_2)$. Iterating this process yields the unique expression 
\begin{equation}
u=y_1\novpromag_{\omega_1}(y_2\novpromag_{{\omega_2}}(\cdots (y_n\novpromag_{\omega_n}z)\cdots)),
\mlabel{eq:rgraft}
\end{equation}
where $z\in X$, $y_i\in \mathcal{T}$ and $\omega_i\in \Omega,i=1,\cdots,n$. 
An element of $X$ can also be regarded as a $u$ in Eq.~\eqref{eq:rgraft} with $n=0$.

For our purpose of explicitly constructing free noncommuting multi-Novikov algebras, we give another construction of the free objects by quotients. 

Let ${\rm MT}_\Omega(X)$ denote the set of all typed decorated rooted trees, defined to be the rooted trees whose vertices are decorated by elements of $X$  and edges are decorated by elements of $\Omega$ \cite{Foi0}. 
See\,\mcite{GGL} for such trees, called vertex-edge decorated trees, in the algebraic study of Volterra integral equations. 
As an algebraic interpretation of the typed decorated rooted trees, we give the following notion\,\mcite{DL}. 

\begin{defi}
An {\bf $\bf{r}$-expression} is recursively given by $r(z)$ for $z\in X$ and by 
\begin{equation}
r(y_1,\omega_1,\cdots,y_n,\omega_n;z),
\mlabel{e:rexp}
\end{equation}
where $z\in X$, $\omega_i\in \Omega$, and each $y_i\in {\rm MT}_\Omega(X)$ itself is an ${r}$-expression, $i=1,\cdots, n$. 
\end{defi}


In the following, the set of typed decorated trees are encoded as the set of $r$-expressions by 
encoding 
$\bullet_z$ as $z$ for $z\in X$, and encoding
$$
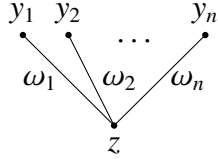
\begin{tikzpicture}[scale=0.6]
\coordinate (z) at (0,0);
\coordinate (y1) at (-2,2);
\coordinate (y2) at (-1,2);
\coordinate (yn) at (2,2);
\draw (z) -- (y1) node[midway, left] {$\omega_1$};
\draw (z) -- (y2) node[midway,right] {$\omega_2$};
\draw (z) -- (yn) node[midway,right] {$\omega_n$};
\fill (-2.8,1.3)  node[below] {};
\fill (y1) circle (2pt) node[above] {$y_1$};
\fill (y2) circle (2pt) node[above] {$y_2$};
\fill (z) circle (2pt) node[below] {$z$};
\fill (yn) circle (2pt) node[above] {$y_n$};
\node at (0.5,1.8) {$\cdots$};
\end{tikzpicture}
$$
as the $r$-expression $r(y_1,\omega_1,\cdots,y_n,\omega_n;z)$. 

\nc{\tdeg}{\mathrm{tdeg}}

The {\bf total degree} $\tdeg (T)$ of a typed decorated tree $T$ is the total number of times (counting multiplicity) that elements of $X\sqcup \Omega$ appearing in the tree. In terms of an $r$-expression $r$, the {\bf total degree} $\tdeg(r)$ it is the total number of times that elements of $X\sqcup \Omega$ appearing in the expression. 

We  recursively define a linear map 
\begin{align}\mlabel{eq:tm1}
\begin{split}
h: \bfk{\rm MT}_\Omega(X)\ra& \fmma, \\
r(z)\mapsto&~ {z}, \\
r(y_1,\omega_1,\cdots,y_n,\omega_n;z)\mapsto&~ h(y_1) \novpromag_{\omega_1} h\big(r(y_2,\omega_2,\cdots,y_n,\omega_n;z)),
\end{split}
\end{align}
for  $z\in X,\omega_i\in \Omega$ with $i=1,\cdots,n$.

\begin{lm}
The linear map $h: \bfk{\rm MT}_\Omega(X)\ra \fmma$ in Eq.\,\meqref{eq:tm1} is bijective.
\mlabel{l:otocbmam}
\end{lm}

\nc{\calt}{\mathcal{T}}

\begin{proof}
To give the proof, we recursively construct another linear map
\begin{align}\mlabel{eq:tm2}
\begin{split}
t: \fmma\ra& \bfk{\rm MT}_\Omega(X),\\
z\mapsto&~ {r(z)},\quad\text{ for } z\in X, \\
u=y_1\novpromag_{\omega_1}\Big(y_2{\novpromag_{\omega_2}}(\cdots (y_n\novpromag_{\omega_n}z)\cdots)\Big) \mapsto &~ r\big(t(y_1) ,{\omega_1},t(y_2),\omega_2, \cdots,t(y_n),\omega_n;z\big),
\end{split}
\end{align}
for a planar binary tree $u\in \mathcal{T}_\Omega(X)$ with the unique factorization in Eq.\meqref{eq:rgraft}. 

We next verify $ht(u)=\id_{\fmma}(u)$ by induction on the total degree $\tdeg(u)$ of $u\in \calt_\Omega(X)$. First 
$ht(z)=h(r(z))=z,$ for $z\in X$. 
Inductively, for $y_1\novpromag_{\omega_1}u_1\in \calt$ with $u_1=y_2{\novpromag_{\omega_2}}(\cdots (y_n\novpromag_{\omega_n}z)\cdots)\in \calt$, we have 
$$ ht(y_1\novpromag_{\omega_1}u_1) = h\big(r\big(t(y_1) ,{\omega_1},t(y_2), \cdots,t(y_n),\omega_n;z\big)\big)
= ht(y_1)\novpromag_{\omega_1} ht(u_1)=y_1\novpromag_{\omega_1} u_1.$$
Similarly, $th=\id_{\bfk{\rm MT}_\Omega(X)}$.
\end{proof}

Through the linear bijection $h$, the free $\Omega$-multi-magmatic algebra structure on $\fmma$ defines a  free $\Omega$-multi-magmatic algebra structure on $\bfk{\rm MT}_\Omega(X)$ by transporting of structures, so that there is an $\Omega$-multi-magmatic algebra isomorphism
$$ t: (\fmma, \{\novpromag_\omega\}_{\omega\in \Omega}) \to (\bfk{\rm MT}_\Omega(X), \{\rhd_\omega\}_{\omega\in\Omega}).$$
Here for $u, v\in {\rm MT}_\Omega(X)$ and $\omega\in \Omega$, define
\vspb
$$ u \rhd_\omega v: = t(h(u)\novpromag_\omega h(v)).$$
So if $v=r(y_1,\omega_1,\cdots,y_n,\omega_n;z)$, then by Eqs.\,\meqref{eq:tm1}--\meqref{eq:tm2}, we have 
\begin{equation}
\mlabel{e:rexpprod}
\begin{split}
u {\rhd}_\omega v: &= t(h(u)\novpromag_\omega h(v))\\
&=t\bigg(h(u)\novpromag_\omega \Big(h(y_1)\novpromag_{\omega_1} \big(h(y_2)\novpromag_{\omega_2} (\cdots (h(y_n)\novpromag_{\omega_n} z))\big)\Big)\bigg)\\
&= r(u,\omega,y_1,\omega_1,\cdots,y_n,\omega_n;z).
\end{split}
\end{equation}

In other words, $u$ is grafted on the root of $v$ (with edge decoration $\omega$) from the left\,~\mcite{Foi}.
For example, 
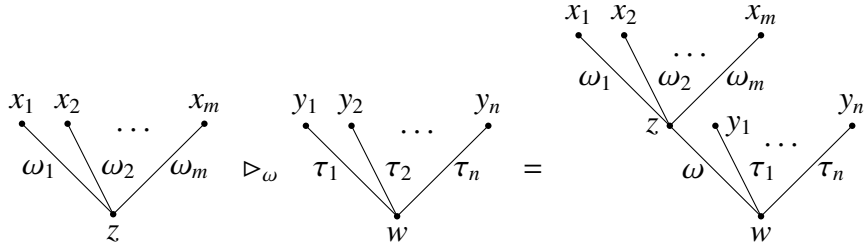
\[
\begin{tikzpicture}[scale=0.6]
\coordinate (z) at (0,0);
\coordinate (y1) at (-2,2);
\coordinate (y2) at (-1,2);
\coordinate (yn) at (2,2);
\draw (z) -- (y1) node[midway, left] {$\omega_1$};
\draw (z) -- (y2) node[midway,right] {$\omega_2$};
\draw (z) -- (yn) node[midway,right] {$\omega_m$};
\fill (-2.8,1.3)  node[below] {};
\fill (y1) circle (2pt) node[above] {$x_1$};
\fill (y2) circle (2pt) node[above] {$x_2$};
\fill (z) circle (2pt) node[below] {$z$};
\fill (yn) circle (2pt) node[above] {$x_m$};
\node at (0.5,1.8) {$\cdots$};
\end{tikzpicture}
\begin{tikzpicture}[scale=0.6]
\draw  (-3,1)  node {${\rhd}_\omega$};
\coordinate (z) at (0,0);
\coordinate (y1) at (-2,2);
\coordinate (y2) at (-1,2);
\coordinate (yn) at (2,2);
\draw (z) -- (y1) node[midway, left] {$\tau_1$};
\draw (z) -- (y2) node[midway,right] {$\tau_2$};
\draw (z) -- (yn) node[midway,right] {$\tau_n$};
\fill (-2.8,1.3)  node[below] {};
\fill (y1) circle (2pt) node[above] {$y_1$};
\fill (y2) circle (2pt) node[above] {$y_2$};
\fill (z) circle (2pt) node[below] {$w$};
\fill (yn) circle (2pt) node[above] {$y_n$};
\node at (0.5,1.8) {$\cdots$};
\end{tikzpicture}
\begin{tikzpicture}[scale=0.6]
\draw  (-3,1)  node {$=$};
\coordinate (z0) at (2,0);
\coordinate (x1) at (0,2);
\coordinate (x2) at (1,2);
\coordinate (xm) at (4,2);
\coordinate (y1) at (-2,4);
\coordinate (y2) at (-1,4);
\coordinate (yn) at (2,4);

\draw (z0) -- (x1) node[midway,left] {$\omega$};
\draw (z0) -- (x2) node[midway,right] {$\tau_1$};
\draw  (2.5,1.5)  node {$\cdots$};
\draw (z0) -- (xm) node[midway,right] {$\tau_n$};
\draw (x1) -- (y1) node[midway,left] {$\omega_1 $};
\draw (x1) -- (y2) node[midway,right] {$\omega_2 $};
\draw  (0.5,3.5)  node {$\cdots$};
\draw (x1) -- (yn) node[midway,right] {$\omega_m $};
\fill (z0) circle (2pt) node[below]{$w$};
\fill (x1) circle (2pt) node[left]{$z$};
\fill (x2) circle (2pt) node[above,right]{$y_1$};
\fill (xm) circle (2pt) node[above]{$y_n$};
\fill (y1) circle (2pt) node[above]{$x_1$};
\fill (y2) circle (2pt) node[above]{$x_2$};
\fill (yn) circle (2pt) node[above]{$x_m$};
\end{tikzpicture}.
\]

Furthermore, the image $J:=t(I)$ of the multi-Novikov two-sided ideal $I$ is a  multi-Novikov two-sided ideal of $\bfk{\rm MT}_\Omega(X)$, and  
there is the following isomorphism of multi-Novikov algebras
\begin{equation}
\widetilde{h}:  \bfk{\rm MT}_\Omega(X)/J\overset{\cong}{\to} \fmma/I,\quad u+J\mapsto h(u)+I.
\mlabel{eq:isofnov}
\end{equation}
Hence
\vspc
\begin{equation}
\fnmna:=\bfk{\rm MT}_\Omega(X)/J \cong \fmma / I
\mlabel{e:qnov}
\end{equation}
is also the free $\Omega$-noncommuting multi-Novikov algebra on $X$.

\begin{remark}
\mlabel{remark:hmn}
For later applications, we interpret the two generators Eq.~\meqref{e:ncmnrel1}--\meqref{e:ncmnrel2} of the ideal $I$ for the quotient algebra $\fmma /I$ in terms of generators of the ideal $J$ for the quotient $\fnmna=\bfk{\rm MT}_\Omega(X)/J $.
\begin{enumerate}
\item 
\mlabel{item:hmn2}
We begin with the simpler case, for the generator Eq.~\eqref{e:ncmnrel2} of $I$, its corresponding generator in $J$ via the isomorphism $t$ in Eq.\,\mref{eq:tm2} is  
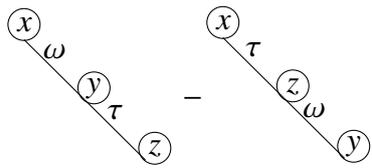
$$
\begin{tikzpicture}[scale=0.4]
\coordinate (z0) at (2,0);
\coordinate (y1) at (0,2);
\coordinate (y2) at (-2,4);

\draw (z0) -- (y1) node[midway,above] {$\tau$};
\draw (y1) -- (y2) node[midway,above] {$\omega $};

\draw  (z0) +(0.3,0.4) circle (15pt) node {$z$};
\draw (y1)+(0.3,0.4) circle (15pt) node {$y$};
\draw (y2)+(0,0.5) circle (15pt) node {$x$};
\end{tikzpicture}
\begin{tikzpicture}[scale=0.4]
\draw  (-3,2)  node {$-$};
\coordinate (z0) at (2,0);
\coordinate (y1) at (0,2);
\coordinate (y2) at (-2,4);

\draw (z0) -- (y1) node[midway,above] {$\omega$};
\draw (y1) -- (y2) node[midway,above] {$\tau $};

\draw  (z0) +(0.3,0.4) circle (15pt) node {$y$};
\draw (y1)+(0.3,0.4) circle (15pt) node {$z$};
\draw (y2)+(0,0.5) circle (15pt) node {$x$};
\end{tikzpicture}.$$
When $x, y, z$ are in $X$, (modulo) this relation means that the two segments of the tree with roots $y$ and $z$ can be exchanges, while the leaf $x$ is fixed. 
More generally, when we take $x=x_1$, $y=x_2\rhd_{\tau_2}(x_3\rhd_{\tau_3}( \cdots(x_{m}\rhd_{\tau_{m}}z'))\cdots)$, 
$z=y_2\rhd_{\omega_2}(y_3\rhd_{\omega_3}( \cdots(y_{n}\rhd_{\omega_{n}}z))\cdots)$, $\omega=\tau_1$ and $\tau=\omega_1$. 
Then for the $r$-expression 
$
r(y_1,\omega_1,\cdots ,y_n,\omega_n;z)$
with $y_1=r(x_1,\tau_1,\cdots,x_m,\tau_m;z')$, we have
\begin{equation}
\begin{split}
&r\big(r(x_1,\tau_1,\cdots,x_m,\tau_m;z'),\omega_1,y_2,\cdots ,y_n,\omega_n;z\big)\\
\equiv&~	r\big(r(x_1,\omega_1,y_2,\cdots ,y_n,\omega_n;z),\tau_1,\cdots,x_m,\tau_m;z'\big)	\mod  J
\end{split}
\mlabel{e:n22}
\end{equation}
In terms of typed trees, we have 
$$
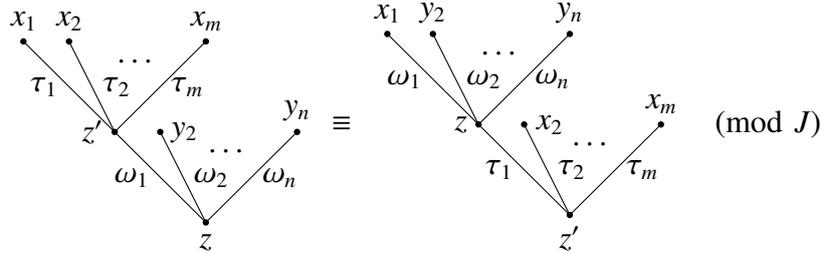
\begin{tikzpicture}[scale=0.6]
\coordinate (z0) at (2,0);
\coordinate (x1) at (0,2);
\coordinate (x2) at (1,2);
\coordinate (xm) at (4,2);
\coordinate (y1) at (-2,4);
\coordinate (y2) at (-1,4);
\coordinate (yn) at (2,4);

\draw (z0) -- (x1) node[midway,left] {$\omega_1$};
\draw (z0) -- (x2) node[midway,right] {$\omega_2$};
\draw  (2.5,1.5)  node {$\cdots$};
\draw (z0) -- (xm) node[midway,right] {$\omega_n$};
\draw (x1) -- (y1) node[midway,left] {$\tau_1 $};
\draw (x1) -- (y2) node[midway,right] {$\tau_2 $};
\draw  (0.5,3.5)  node {$\cdots$};
\draw (x1) -- (yn) node[midway,right] {$\tau_m $};
\fill (z0) circle (2pt) node[below]{$z$};
\fill (x1) circle (2pt) node[left]{$z'$};
\fill (x2) circle (2pt) node[above,right]{$y_2$};
\fill (xm) circle (2pt) node[above]{$y_n$};
\fill (y1) circle (2pt) node[above]{$x_1$};
\fill (y2) circle (2pt) node[above]{$x_2$};
\fill (yn) circle (2pt) node[above]{$x_m$};
\end{tikzpicture}
\begin{tikzpicture}[scale=0.6]
\draw  (-3,2)  node {$\equiv$};
\coordinate (z0) at (2,0);
\coordinate (x1) at (0,2);
\coordinate (x2) at (1,2);
\coordinate (xm) at (4,2);
\coordinate (y1) at (-2,4);
\coordinate (y2) at (-1,4);
\coordinate (yn) at (2,4);

\draw (z0) -- (x1) node[midway,left] {$\tau_1$};
\draw (z0) -- (x2) node[midway,right] {$\tau_2$};
\draw  (2.5,1.5)  node {$\cdots$};
\draw (z0) -- (xm) node[midway,right] {$\tau_m$};
\draw (x1) -- (y1) node[midway,left] {$\omega_1 $};
\draw (x1) -- (y2) node[midway,right] {$\omega_2 $};
\draw  (0.5,3.5)  node {$\cdots$};
\draw (x1) -- (yn) node[midway,right] {$\omega_n $};
\fill (z0) circle (2pt) node[below]{$z'$};
\fill (x1) circle (2pt) node[left]{$z$};
\fill (x2) circle (2pt) node[above,right]{$x_2$};
\fill (xm) circle (2pt) node[above]{$x_m$};
\fill (y1) circle (2pt) node[above]{$x_1$};
\fill (y2) circle (2pt) node[above]{$y_2$};
\fill (yn) circle (2pt) node[above]{$y_n$};
\draw  (6,2)  node {$\pmod J$};	
\end{tikzpicture}$$
indicating that the two adjacent segments with roots $z$ and $z'$ can be exchanged, while the left-most leaf is fixed. 

\item Under the isomorphism $t$, the generator Eq.~\eqref{e:ncmnrel1} of  $I$ corresponds to the following generator of $J$  
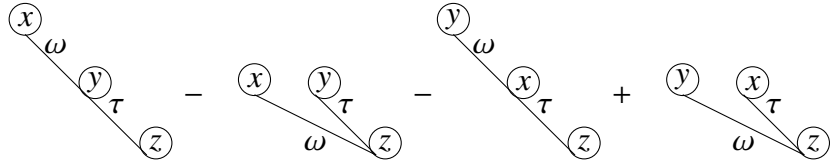
$$
\begin{tikzpicture}[scale=0.4]
\coordinate (z0) at (2,0);
\coordinate (y1) at (0,2);
\coordinate (y2) at (-2,4);

\draw (z0) -- (y1) node[midway,above] {$\tau$};
\draw (y1) -- (y2) node[midway,above] {$\omega $};

\draw  (z0) +(0.3,0.4) circle (15pt) node {$z$};
\draw (y1)+(0.3,0.4) circle (15pt) node {$y$};
\draw (y2)+(0,0.5) circle (15pt) node {$x$};
\end{tikzpicture}
\begin{tikzpicture}[scale=0.4]
\draw  (-6,2)  node {$-$};
\coordinate (z0) at (0,0);
\coordinate (y1) at (-2,2);
\coordinate (y2) at (-4,2);

\draw (z0) -- (y1) node[midway,above] {$\tau$};
\draw (z0) -- (y2) node[midway,below] {$\omega $};

\draw  (z0) +(0.3,0.4) circle (15pt) node {$z$};
\draw (y1)+(0.3,0.4) circle (15pt) node {$y$};
\draw (y2)+(0,0.5) circle (15pt) node {$x$};
\end{tikzpicture}
\begin{tikzpicture}[scale=0.4]
\draw  (-3,2)  node {$-$};
\coordinate (z0) at (2,0);
\coordinate (y1) at (0,2);
\coordinate (y2) at (-2,4);

\draw (z0) -- (y1) node[midway,above] {$\tau$};
\draw (y1) -- (y2) node[midway,above] {$\omega $};

\draw  (z0) +(0.3,0.4) circle (15pt) node {$z$};
\draw (y1)+(0.3,0.4) circle (15pt) node {$x$};
\draw (y2)+(0,0.5) circle (15pt) node {$y$};
\end{tikzpicture}
\begin{tikzpicture}[scale=0.4]
\draw  (-6,2)  node {$+$};
\coordinate (z0) at (0,0);
\coordinate (y1) at (-2,2);
\coordinate (y2) at (-4,2);

\draw (z0) -- (y1) node[midway,above] {$\tau$};
\draw (z0) -- (y2) node[midway,below] {$\omega $};

\draw  (z0) +(0.3,0.4) circle (15pt) node {$z$};
\draw (y1)+(0.3,0.4) circle (15pt) node {$x$};
\draw (y2)+(0,0.5) circle (15pt) node {$y$};
\end{tikzpicture}
$$
Regard it as a reduction rule:
$$
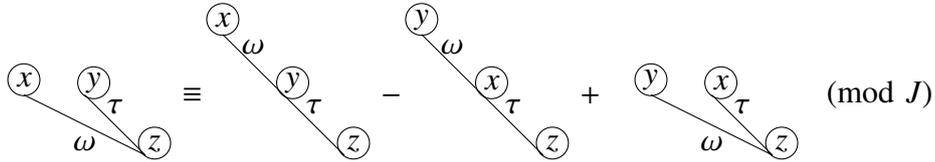
\begin{tikzpicture}[scale=0.4]
\coordinate (z0) at (0,0);
\coordinate (y1) at (-2,2);
\coordinate (y2) at (-4,2);

\draw (z0) -- (y1) node[midway,above] {$\tau$};
\draw (z0) -- (y2) node[midway,below] {$\omega $};

\draw  (z0) +(0.3,0.4) circle (15pt) node {$z$};
\draw (y1)+(0.3,0.4) circle (15pt) node {$y$};
\draw (y2)+(0,0.5) circle (15pt) node {$x$};
\end{tikzpicture}
\begin{tikzpicture}[scale=0.4]
\draw  (-3,2)  node {$\equiv$};
\coordinate (z0) at (2,0);
\coordinate (y1) at (0,2);
\coordinate (y2) at (-2,4);

\draw (z0) -- (y1) node[midway,above] {$\tau$};
\draw (y1) -- (y2) node[midway,above] {$\omega $};

\draw  (z0) +(0.3,0.4) circle (15pt) node {$z$};
\draw (y1)+(0.3,0.4) circle (15pt) node {$y$};
\draw (y2)+(0,0.5) circle (15pt) node {$x$};
\end{tikzpicture}
\begin{tikzpicture}[scale=0.4]
\draw  (-3,2)  node {$-$};
\coordinate (z0) at (2,0);
\coordinate (y1) at (0,2);
\coordinate (y2) at (-2,4);

\draw (z0) -- (y1) node[midway,above] {$\tau$};
\draw (y1) -- (y2) node[midway,above] {$\omega $};

\draw  (z0) +(0.3,0.4) circle (15pt) node {$z$};
\draw (y1)+(0.3,0.4) circle (15pt) node {$x$};
\draw (y2)+(0,0.5) circle (15pt) node {$y$};
\end{tikzpicture}
\begin{tikzpicture}[scale=0.4]
\draw  (-6,2)  node {$+$};
\coordinate (z0) at (0,0);
\coordinate (y1) at (-2,2);
\coordinate (y2) at (-4,2);

\draw (z0) -- (y1) node[midway,above] {$\tau$};
\draw (z0) -- (y2) node[midway,below] {$\omega $};

\draw  (z0) +(0.3,0.4) circle (15pt) node {$z$};
\draw (y1)+(0.3,0.4) circle (15pt) node {$x$};
\draw (y2)+(0,0.5) circle (15pt) node {$y$};
\draw  (3,2)  node {$\pmod J$};	
\end{tikzpicture}
$$
The intuitive meaning is that the two branches of the left tree can be exchanges at the cost of introducing two extra trees with fewer branches. 

More generally, for a fixed $i\geq 1$, apply the congruence to $x=y_i$, $y=y_{i+1}$, $z=y_{i+2}\rhd_{\omega_{i+2}}(y_{i+3}\rhd_{\omega_{i+3}}( \cdots(y_{n}\rhd_{\omega_{n}}z))\cdots)$, $\omega=\omega_i$ and $\tau=\omega_{i+1}$ with $1\leq  i\leq n-1$.
Then multiply both sides of the resulting congruence on the left by the element  
$y_{i-1}$ via the multiplication $\rhd_{\omega_{i-1}}$. Then multiply both sides of the new resulting congruence on the left by  
$y_{i-2}$ via  $\rhd_{\omega_{i-2}}$. Then continue this process repeatedly until multiplying both sides on the left by  
$y_{1}$ via $\rhd_{\omega_{1}}$.
We finally obtain the congruence
\begin{equation}
\begin{split}
&~r(y_1,\omega_1,\cdots ,y_{i},\omega_{i},y_{i+1},\omega_{i+1},\cdots,y_n,\omega_n;z)\\
\equiv&~r\big(y_1,\omega_1,\cdots y_{i-1},\omega_{i-1}, y_{i}\rhd_{\omega_{i}}y_{i+1},\omega_{i+1},\cdots,y_n,\omega_n;z\big)\\
&-r\big(y_1,\omega_1,\cdots y_{i-1},\omega_{i-1}, y_{i+1}\rhd_{\omega_{i}}y_{i},\omega_{i+1},\cdots,y_n,\omega_n;z\big)\\
&+r(y_1,\omega_1,\cdots ,y_{i-1},\omega_{i-1},y_{i+1},\omega_{i},y_{i},\omega_{i+1},\cdots,y_n,\omega_n;z) ~~\mod{J}.
\end{split}
\mlabel{e:n21}
\end{equation}
In terms of typed decorated trees, we have 
$$
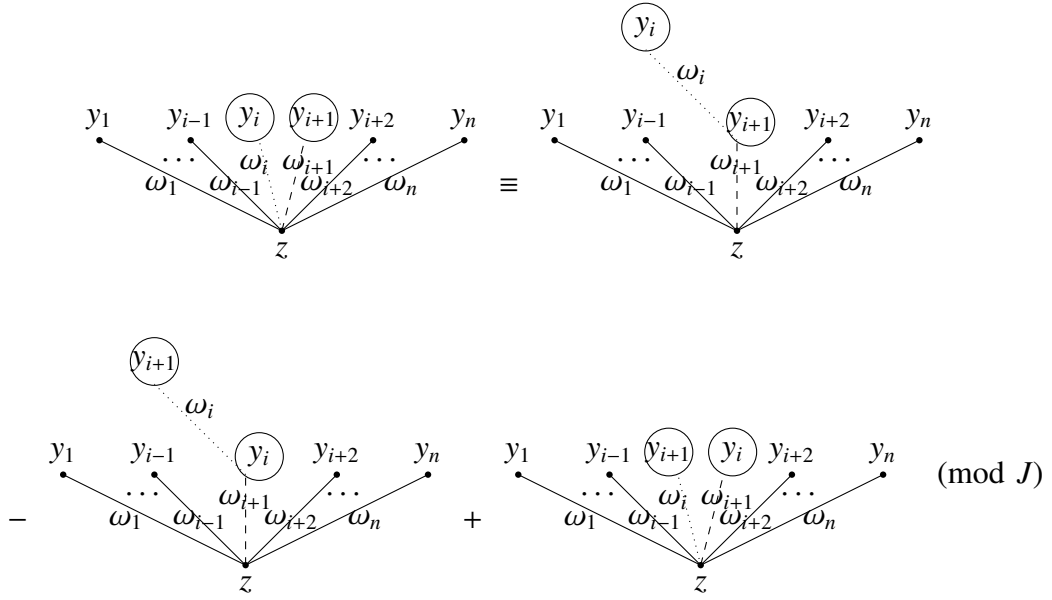
\begin{tikzpicture}[scale=0.6]
\coordinate (z0) at (0,0);
\coordinate (w1) at (-4,2);
\coordinate (w2) at (-2,2);
\coordinate (y1) at (-0.5,2);
\coordinate (y2) at (0.5,2);
\coordinate (z1) at (2,2);
\coordinate (z2) at (4,2);
\draw (z0) -- (w1) node[midway,left] {$\omega_1$};
\draw  (-2.2,1.5)  node {$\cdots$};
\draw  (z0) -- (w2) node[midway] {$\omega_{i-1}$};
\draw[dotted] (z0) -- (y1) node[midway,above] {$\omega_i\quad$};
\draw[dashed] (z0) -- (y2) node[midway,above] {$\quad \omega_{i+1}$};
\draw (z0) -- (z1) node[midway] {$\omega_{i+2}$};
\draw  (2.2,1.5)  node {$\cdots$};
\draw (z0) -- (z2) node[midway,right] {$\omega_{n}$};

\fill (z0) circle (2pt) node[below] {$z$};
\fill (w1) circle (2pt) node[above] {$y_1$};
\fill(w2)circle (2pt) node[above] {$y_{i-1}$};
\draw (y1)+(-0.2,0.5) circle (15pt) node {$y_i$};
\draw (y2)+(0.2,0.5) circle (15pt) node {$y_{i+1}$};
\fill (z1) circle (2pt) node[above]{$y_{i+2}$};
\fill (z2) circle (2pt) node[above]{$y_n$};
\end{tikzpicture}
\begin{tikzpicture}[scale=0.6]
\draw  (-5,1)  node {$\equiv$};
\coordinate (z0) at (0,0);
\coordinate (w1) at (-4,2);
\coordinate (w2) at (-2,2);
\coordinate (y1) at (0,2);
\coordinate (y2) at (-2,4);
\coordinate (z1) at (2,2);
\coordinate (z2) at (4,2);

\draw (z0) -- (w1) node[midway,left] {$\omega_1$};
\draw  (-2.2,1.5)  node {$\cdots$};
\draw  (z0) -- (w2) node[midway] {$\omega_{i-1}$};
\draw[dashed] (z0) -- (y1) node[midway,above] {$\omega_{i+1}$};
\draw[dotted] (y1) -- (y2) node[midway,above] {$\omega_i $};
\draw (z0) -- (z1) node[midway] {$\omega_{i+2}$};
\draw  (2.2,1.5)  node {$\cdots$};
\draw (z0) -- (z2) node[midway,right] {$\omega_{n}$};

\fill (z0) circle (2pt) node[below] {$z$};

\fill(w1)  circle (2pt) node[above] {$y_1$};
\fill (w2)  circle (2pt) node[above] {$y_{i-1}$};
\draw (y1)+(0.3,0.4) circle (15pt) node {$y_{i+1}$};
\draw (y2)+(0,0.5) circle (15pt) node {$y_{i}$};
\fill (z1) circle (2pt) node[above] {$y_{i+2}$};
\fill (z2) circle (2pt) node[above] {$y_n$};
\end{tikzpicture}
$$

$$
\quad
\begin{tikzpicture}[scale=0.6]
\draw  (-5,1)  node {$-$};
\coordinate (z0) at (0,0);
\coordinate (w1) at (-4,2);
\coordinate (w2) at (-2,2);
\coordinate (y1) at (0,2);
\coordinate (y2) at (-2,4);
\coordinate (z1) at (2,2);
\coordinate (z2) at (4,2);

\draw (z0) -- (w1) node[midway,left] {$\omega_1$};
\draw  (-2.2,1.5)  node {$\cdots$};
\draw  (z0) -- (w2) node[midway] {$\omega_{i-1}$};
\draw[dashed] (z0) -- (y1) node[midway,above] {$\omega_{i+1}$};
\draw[dotted] (y1) -- (y2) node[midway,above] {$\omega_i $};
\draw (z0) -- (z1) node[midway] {$\omega_{i+2}$};
\draw  (2.2,1.5)  node {$\cdots$};
\draw (z0) -- (z2) node[midway,right] {$\omega_{n}$};

\fill (z0) circle (2pt) node[below] {$z$};

\fill(w1)  circle (2pt) node[above] {$y_1$};
\fill (w2)  circle (2pt) node[above] {$y_{i-1}$};
\draw (y1)+(0.3,0.4) circle (15pt) node {$y_{i}$};
\draw (y2)+(0,0.5) circle (15pt) node {$y_{i+1}$};
\fill (z1) circle (2pt) node[above] {$y_{i+2}$};
\fill (z2) circle (2pt) node[above] {$y_n$};
\end{tikzpicture}
\begin{tikzpicture}[scale=0.6]
\draw  (-5,1)  node {$+$};
\coordinate (z0) at (0,0);
\coordinate (w1) at (-4,2);
\coordinate (w2) at (-2,2);
\coordinate (y1) at (-0.5,2);
\coordinate (y2) at (0.5,2);
\coordinate (z1) at (2,2);
\coordinate (z2) at (4,2);
\draw (z0) -- (w1) node[midway,left] {$\omega_1$};
\draw  (-2.2,1.5)  node {$\cdots$};
\draw  (z0) -- (w2) node[midway] {$\omega_{i-1}$};
\draw[dotted] (z0) -- (y1) node[midway,above] {$\omega_i\quad$};
\draw[dashed] (z0) -- (y2) node[midway,above] {$\quad \omega_{i+1}$};
\draw (z0) -- (z1) node[midway] {$\omega_{i+2}$};
\draw  (2.2,1.5)  node {$\cdots$};
\draw (z0) -- (z2) node[midway,right] {$\omega_{n}$};

\fill (z0) circle (2pt) node[below] {$z$};
\fill (w1) circle (2pt) node[above] {$y_1$};
\fill(w2)circle (2pt) node[above] {$y_{i-1}$};
\draw (y1)+(-0.2,0.5) circle (15pt) node {$y_{i+1}$};
\draw (y2)+(0.2,0.5) circle (15pt) node {$y_{i}$};
\fill (z1) circle (2pt) node[above]{$y_{i+2}$};
\fill (z2) circle (2pt) node[above]{$y_n$};
\draw  (6,2)  node {$\pmod J$};	
\end{tikzpicture}
$$
\mlabel{item:hmn1}
\end{enumerate}
This again has the combinatorial interpretation that the two branches for $y_i$ and $y_{i+1}$ can be exchanged at the cost of introducing two new trees with newer branches. 
\end{remark}

In summary, we give the following construction and relations for the free $\Omega$-multi-Novikov algebra on $X$ for later application. 

\begin{prop}
\begin{enumerate}
\item The quotient $\fnmna:=\bfk{\rm MT}_\Omega(X)/J$ equipped with multiplications ${\rhd}_\omega, \omega\in \Omega$, is the free $\Omega$-multi-Novikov algebra on $X$. 
\item 
Eqs.\,\meqref{e:n22}--\meqref{e:n21} hold in $\fnmna=\bfk{\rm MT}_\Omega(X)/J$. 
\end{enumerate}
\mlabel{p:rexpreln}	
\end{prop}

\vspd
\subsection{Statement of the main theorem}
\mlabel{ss:main}
For given sets $X$ and $\Omega$, we first state Theorem\,\mref{t:fncmncasubnov} that gives two explicit constructions of the free $\Omega$-noncommuting multi-Novikov algebras on $X$. 

\newcommand{\tildeh}{\eta}

The readers are invited to use the following diagram to keep track of the notions and maps in the constructions. 
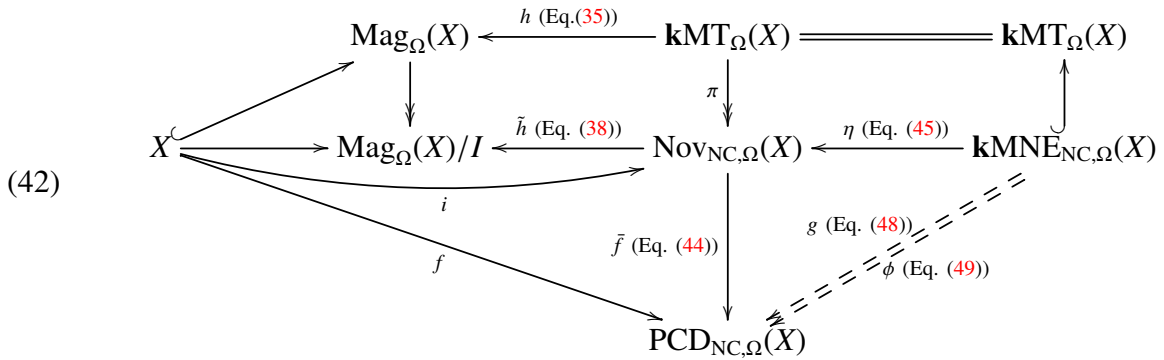
\begin{equation}
\begin{split}
\xymatrix{
&& \fmma \ar@{->>}[d] &&\ar[ll]_{h~({\rm Eq.}\eqref{eq:tm1})}\bfk{\rm MT}_\Omega(X)\ar@{->>}[d]_{\pi}&& \ar@{=}[ll]\bfk{\rm MT}_\Omega(X) \\
\Omegax  \ar@{^(->}[rru] \ar[rr]  \ar[ddrrrr]_f  \ar@/_15pt/[rrrr]_{i} & & \fmma/I  &&\ar[ll]_{\tilde{h}~({\rm Eq.}~\eqref{eq:isofnov})}\fnmna\ar[dd]_{\bar{f}~({\rm Eq.}~\eqref{eq:barf})} && \bfk\NMNE       \ar[ll]_{\tildeh~({\rm Eq.}~\eqref{eq:tildeh})} \ar@{^(->}[u]  \ar@{-->}[ddll]_{g~({\rm Eq.}~\eqref{eq:mtop})} \ar@<.5ex>@{-->}[ddll]^{\phi~({\rm Eq.}~\eqref{e:phi})}  \\
&&&&\\
& &       &&\fnmdca& }
\end{split}
\mlabel{e:diag}
\end{equation}
The dashed arrows $\phi$ and $g$ are only needed in the proof of Theorem~\mref{t:fncmncasubnov}. 

For the first construction of the free $\Omega$-multi-Novikov algebra on $X$, we identify a special class of $r$-expressions. 

\begin{defi}
Let  $u=r(y_1,\omega_1,\cdots,y_n,\omega_n;z)$ be an $r$-expression in ${\rm MT}_\Omega(X)$. 
\begin{enumerate}
\item $u$ is called a {\bf nest} if  $u\in \Omegax$ or $y_i\in \Omegax, i\geq 2$ and the $r$-expression $y_1$  itself is a nest.
\mlabel{i:re1}
\item Let $\Omegax$ and $\Omega$ be two sets with total orders. A nest $u$ is called {\bf ordered} if $u$ satisfies one of the following conditions
\begin{enumerate}
\item $n=0$,  in which case $u\in \Omegax$ , 
\item $n>0$   and  $ y_1\in X $         , 
\item $n>0,  y_1=r(y'_1,\omega'_1,\cdots,y'_m,\omega'_m;z')$ and  $m>n,$
\item $n>0,  y_1=r(y'_1,\omega'_1,\cdots,y'_m,\omega'_m;z'),  m=n $ and  $ z'> z,$
\item $n>0,  y_1=r(y'_1,\omega'_1,\cdots,y'_m,\omega'_m;z'),  m=n, z'= z $ and  $(\omega_1',\cdots,\omega_n' )\geq (\omega_1,\cdots,\omega_n )$ with respect to the lexicographical order.
\end{enumerate}

\mlabel{i:re2}
\item An ordered nest $u=r(y_1,\omega_1,\cdots,y_n,\omega_n;z)\in{\rm MT}_\Omega(X)$ is called a {\bf noncommuting multi-Novikov element} if the sequence of all of its leaves, reading from the right is in increasing order with respect to the order of $X$.
Denote the set of noncommuting multi-Novikov elements by $\NMNE.$
\mlabel{i:re3}
\end{enumerate}
\mlabel{d:re}
\end{defi}

\begin{remark}\mlabel{rm:notorder}
The five cases in Definition \mref{d:re}\ref{i:re2} are depicted in the left-hand side of the flow diagram in Figure~\ref{fg:notorder}, while the three dashed boxes on the right-hand side indicate that $u=r(y_1,\omega_1,\cdots,y_n,\omega_n;z) $ is not ordered. Therefore, if $u$ is not ordered, then $u$ is governed by one of the three dashed boxes.

\begin{figure}
\centering
$$\begin{forest}
for tree={
draw=black,  
rectangle,
rounded corners=2pt, 
align=center,
inner sep=3pt,
l sep=12pt
}
[
{$n\geq 0 $}
[{(i). $n= 0 $}]
[{$n> 0$ }
[{(ii). $y_1\in X $}]
[{$y_1\notin X $ write $y_1=r(y'_1,\omega'_1,\cdots,y'_m,\omega'_m;z')$}
[{(iii). $m>n$}]
[{$m=n$}
[{(iv). $z'>z $}]
[{ $z'=z $}
[{(v). $(\omega_1',\cdots,\omega_n' )\geq (\omega_1,\cdots,\omega_n )$}]
[{ $(\omega_1',\cdots,\omega_n' )< (\omega_1,\cdots,\omega_n )$}, dashed]
]
[{ $z'<z $}, dashed]
]
[{$m<n$}, dashed ]
]
]
]
\end{forest}$$
\caption{}\label{fg:notorder}
\end{figure}
\end{remark}

\begin{ex}
We give the following examples to illustrate each of the cases in ~Definition \mref{d:re}. Let $a<b<c<d$ in $X$, and $\alpha<\beta<\gamma\in \Omega$. 
\begin{enumerate} \item 
$r(r(d,\alpha,a,\gamma,a,\alpha;b),\alpha,c,\beta,a,\alpha;d)$ is a nest.
\[  \begin{tikzpicture}[scale=0.4]
\coordinate (z) at (0,0);
\coordinate (y1) at (2,0);
\coordinate (y2) at (4,0);
\coordinate (y3) at (2,-2);
\coordinate (y4) at (2,2);
\coordinate (y5) at (0,2);
\coordinate (y6) at (-2,2);
\draw (z) -- (y4) node[midway,right] {${_\alpha}\,$};
\draw (z) -- (y5) node[midway,right] {$_\gamma$};
\draw (z) -- (y6) node[midway,left] {$_\alpha$};
\draw (y3) -- (y1) node[midway,right] {$_\beta$};
\draw (y3) -- (y2) node[midway,right] {$_\alpha$};
\draw (y3) -- (z) node[midway,right] {$_\alpha$};
\fill (z) circle (2pt) node[below] {$b$};
\fill (y1) circle (2pt) node[above] {$c$};
\fill (y2) circle (2pt) node[above] {$a$};
\fill (y3) circle (2pt) node[below] {$d$};
\fill (y4) circle (2pt) node[above] {$a$};
\fill (y5) circle (2pt) node[above] {$a$};
\fill (y6) circle (2pt) node[above] {$d$};
\end{tikzpicture}\]
\item 
\begin{enumerate}
\item $x\in X$ is an ordered nest;
\item $r(a,\alpha,c,\beta,a,\beta,d,\gamma;c)$ is an ordered nest;
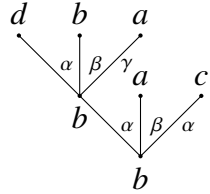
\[  \begin{tikzpicture}[scale=0.4]
\coordinate (z) at (0,0);
\coordinate (y1) at (-3,2);
\coordinate (y2) at (-1,2);
\coordinate (y3) at (1,2);
\coordinate (y4) at (3,2);
\draw (z) -- (y1) node[midway,left] {$_\alpha\,$};
\draw (z) -- (y2) node[midway,left] {$_\beta$};
\draw (z) -- (y3) node[midway,right] {$_\beta$};
\draw (z) -- (y4) node[midway,right] {$_\gamma$};
\fill (z) circle (2pt) node[below] {$c$};
\fill (y1) circle (2pt) node[above] {$a$};
\fill (y2) circle (2pt) node[above] {$c$};
\fill (y3) circle (2pt) node[above] {$a$};
\fill (y4) circle (2pt) node[above] {$d$};
\end{tikzpicture}\]
\item $r(r(a,\gamma,d,\alpha,c,\beta,a,\beta;d),\gamma,a,\alpha;c)$ is an ordered nest;
\[  \begin{tikzpicture}[scale=0.4]
\coordinate (z) at (0,0);
\coordinate (y1) at (-3,2);
\coordinate (y2) at (-1,2);
\coordinate (y3) at (1,2);
\coordinate (y4) at (3,2);
\coordinate (y6) at (4,0);
\coordinate (y7) at (2,-2);
\draw (z) -- (y1) node[midway,left] {$_\gamma\,$};
\draw (z) -- (y2) node[midway,left] {$_\alpha$};
\draw (z) -- (y3) node[midway,right] {$_\beta$};
\draw (z) -- (y4) node[midway,right] {$_\beta$};
\draw (y7) -- (z) node[midway,right] {$_\gamma$};
\draw (y7) -- (y6) node[midway,right] {$_\alpha$};
\fill (z) circle (2pt) node[below] {$d$};
\fill (y1) circle (2pt) node[above] {$a$};
\fill (y2) circle (2pt) node[above] {$d$};
\fill (y3) circle (2pt) node[above] {$c$};
\fill (y4) circle (2pt) node[above] {$a$};
\fill (y6) circle (2pt) node[above] {$a$};
\fill (y7) circle (2pt) node[below] {$c$};
\end{tikzpicture}\]
\item 
$r(r(d,\alpha,b,\beta,a,\gamma;d),\gamma,a,\beta ,c,\alpha;b)$ is an ordered nest;
\[  \begin{tikzpicture}[scale=0.4]
\coordinate (z) at (0,0);
\coordinate (y1) at (-2,2);
\coordinate (y2) at (0,2);
\coordinate (y3) at (2,2);
\coordinate (y4) at (2,0);
\coordinate (y5) at (4,0);
\coordinate (y6) at (2,-2);
\draw (z) -- (y1) node[midway,right] {$_\alpha\,$};
\draw (z) -- (y2) node[midway,right] {$_\beta$};
\draw (z) -- (y3) node[midway,right] {$_\gamma$};
\draw (y6) -- (z) node[midway,right] {$_\gamma$};
\draw (y6) -- (y4) node[midway,right] {$_\beta$};
\draw (y6) -- (y5) node[midway,right] {$_\alpha$};
\fill (z) circle (2pt) node[below] {$d$};
\fill (y1) circle (2pt) node[above] {$d$};
\fill (y2) circle (2pt) node[above] {$b$};
\fill (y3) circle (2pt) node[above] {$a$};
\fill (y4) circle (2pt) node[above] {$a$};
\fill (y5) circle (2pt) node[above] {$c$};
\fill (y6) circle (2pt) node[below] {$b$};
\end{tikzpicture}\]
\item 
$r(r(d,\alpha,b,\beta,a,\gamma;b),\alpha,a,\beta ,c,\alpha;b)$ is an ordered nest.
\[  \begin{tikzpicture}[scale=0.4]
\coordinate (z) at (0,0);
\coordinate (y1) at (-2,2);
\coordinate (y2) at (0,2);
\coordinate (y3) at (2,2);
\coordinate (y4) at (2,0);
\coordinate (y5) at (4,0);
\coordinate (y6) at (2,-2);
\draw (z) -- (y1) node[midway,right] {$_\alpha\,$};
\draw (z) -- (y2) node[midway,right] {$_\beta$};
\draw (z) -- (y3) node[midway,right] {$_\gamma$};
\draw (y6) -- (z) node[midway,right] {$_\alpha$};
\draw (y6) -- (y4) node[midway,right] {$_\beta$};
\draw (y6) -- (y5) node[midway,right] {$_\alpha$};
\fill (z) circle (2pt) node[below] {$b$};
\fill (y1) circle (2pt) node[above] {$d$};
\fill (y2) circle (2pt) node[above] {$b$};
\fill (y3) circle (2pt) node[above] {$a$};
\fill (y4) circle (2pt) node[above] {$a$};
\fill (y5) circle (2pt) node[above] {$c$};
\fill (y6) circle (2pt) node[below] {$b$};
\end{tikzpicture}\]
\end{enumerate}
\item $r(r(r(d,\beta,d,\alpha,c,\gamma,c,\beta;c),\beta,c,\beta,b,\alpha;d),\gamma,a,\beta;a)$ is a noncommuting multi-Novikov element.
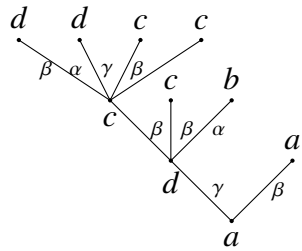
\[  \begin{tikzpicture}[scale=0.4]
\coordinate (z) at (0,0);
\coordinate (y1) at (0,2);
\coordinate (y4) at (2,2);
\coordinate (y5) at (4,0);
\coordinate (y2) at (2,-2);
\coordinate (y3) at (1,4);
\coordinate (yn) at (-2,2);
\coordinate (x) at (-1,4);
\coordinate (z1) at (-3,4);
\coordinate (w) at (-5,4);
\draw (z) -- (y1) node[midway,right] {$_\beta\,$};
\draw (z) -- (y4) node[midway,right] {$_\alpha$};
\draw (z) -- (yn) node[midway,right] {$_\beta$};
\draw (yn) -- (y3) node[midway,left] {$_\beta$};
\draw (yn) -- (x) node[midway,left] {$_\gamma$};
\draw (yn) -- (z1) node[midway,left] {$_\alpha$};
\draw (yn) -- (w) node[midway,left] {$_\beta$};
\draw (y2) -- (y5) node[midway,right] {$_\beta$};
\draw (y2) -- (z) node[midway,right] {$_\gamma$};
\fill (z) circle (2pt) node[below] {$d$};
\fill (y1) circle (2pt) node[above] {$c$};
\fill (y4) circle (2pt) node[above] {$b$};
\fill (y5) circle (2pt) node[above] {$a$};
\fill (y2) circle (2pt) node[below] {$a$};
\fill (y3) circle (2pt) node[above] {$c$};
\fill (yn) circle (2pt) node[below] {$c$};
\fill (x) circle (2pt) node[above] {$c$};
\fill (z1) circle (2pt) node[above] {$d$};
\fill (w) circle (2pt) node[above] {$d$};
\end{tikzpicture}\]
\end{enumerate}
\mlabel{g:nmne}
\end{ex}
\vspc

In prepare for the second construction of the free $\Omega$-multi-Novikov algebra on $X$,  recall from Theorem~\mref{t:fncmdca} that the free $\Omega$-\ncmdca $\bfk_{{\rm NC},\Omega}\{X\}$ generated by $X$ is the commutative polynomial algebra on the set of differential monomials $(\partial_{\omega_1}\cdots \partial_{\omega_n})(x)$, where $\omega_1,\ldots,\omega_n\in \Omega$, $n\in \N$, $x\in X$. 
By Theorem~\mref{p:mdaind}.\mref{p:mnovfromncmdca}, $\bfk_{{\rm NC},\Omega}\{X\}$ is a noncommuting multi-Novikov algebra with the multiplications (here we will use $\ast_\omega$ instead of $\rhd_\omega$ to avoid conflict of notations)
\begin{equation}
x \ast_\omega y:=x\partial_\omega (y),\quad \omega \in\Omega.
\mlabel{eq:indnov}
\end{equation}

Let $\fnmdca$ denote the noncommuting multi-Novikov subalgebra of $\fnmdc$ generated by $X$. 
Since $\fnmna$ is the free noncommuting multi-Novkov algebra generated by $X$, the natural inclusion $f:X\to \fnmdca$ extends to a unique surjection of noncommuting multi-Novkov algebras 
\begin{equation}
\bar{f}:\fnmna \to \fnmdca,
\mlabel{eq:barf}
\end{equation}
which will be shown to be bijective. 

Define a linear map
\begin{align}
\begin{split}
\tildeh: \bfk\NMNE\hookrightarrow \bfk{\rm MT}_\Omega(X)\twoheadrightarrow&~   \fnmna=\bfk{\rm MT}_\Omega(X)/J\\
r(z)\mapsto&~ r({z})+J\equiv r({z})~\mod{J},\\ 
r(y_1,\omega_1,\cdots,y_n,\omega_n;z)\mapsto&~ r(y_1,\omega_1,\cdots,y_n,\omega_n;z)+J\\
&~~~~~\equiv y_1 \rhd_{\omega_1} r(y_2,\omega_2,\cdots,y_n,\omega_n;z)~\mod{J}\\
&~~~~~\equiv  \tildeh(y_1) \rhd_{\omega_1} \tildeh\big(r(y_2,\omega_2,\cdots,y_n,\omega_n;z))\\
\end{split} \mlabel{eq:tildeh}
\end{align}
for  $z,y_i \in X,i=2,\cdots,n,\omega_i\in \Omega$,
where $\rhd_{\omega_1}$ is the multi-Novikov product on $ \fnmna$.

\begin{thm} $(${\bf Main theorem}$)$
\begin{enumerate}
\item 
The linear map $\tildeh: \bfk \NMNE \to  \fnmna$ is bijective. 
\mlabel{i:main1}
\item 
The cosets modulo $J$ of the noncommuting multi-Novikov elements in $\NMNE$ form a linear basis for the free noncommuting multi-Novikov algebra $\fnmna$.
\mlabel{i:main2}
\item
The noncommuting multi-Novikov algebra homomorphism $\bar{f}:\fnmna \to \fnmdca$ in Eq.\,\meqref{eq:barf} is an isomorphism. 
\mlabel{i:main3}
\item 
The space $\fnmdca$, together with operations $\ast_\omega$ in Eq.\,\meqref{eq:indnov} for $\omega \in\Omega$, is a free noncommuting multi-Novikov algebra on $X$.    
\mlabel{i:main4}
\end{enumerate}
\mlabel{t:fncmncasubnov}		
\end{thm}

The proof of the theorem is given in Section~\mref{ss:proof} after preparational results in Section~\mref{ss:prep}, again referring to Diagram \mref{e:diag} for the related notions. In a nutshell, the proof (of Item\,\mref{i:main1}) consists of the following steps. 
\begin{enumerate}
\item We first prove that $\tildeh$ is surjective (Proposition\,\mref{p:fnmnsbmne}). Then together with the surjectivity of $\bar{f}$, the composition $\bar{f}\tildeh:\bfk\NMNE \ra \fnmdca$ is surjective. 
\item To prove that $\tildeh$ is injective, we show that the composition $\bar{f}\tildeh$ is injective. For this purpose, we give an independent description of $\bar{f}\tildeh$ as the map $g$ defined in Eq.\,\eqref{eq:mtop}. See Proposition\,\mref{p:fnnov}. 
\item To prove the injectivity of $g$, we use a pigeonhole principle argument. More precisely, we show that the surjectivity of $g$ as the composition $\bar{f}\tildeh$ is preserved on the finite-dimensional homogeneous components. We then introduce another graded linear map 
$$\phi:\bfk\NMNE \ra \fnmdca$$ 
which is bijective (Proposition\,\mref{p:nmneipcd}). Thus the surjection $g$ on homogeneous components is between vector spaces of the same dimension and hence must be injective.  
\end{enumerate}

\subsection{Preparations}
\mlabel{ss:prep}
We present the preliminary results as outlined above. 

\subsubsection{The surjectivity of $\tildeh$}
\mlabel{sss:surjg}
For a given $r$-expression $u=r(y_1,\omega_1,\cdots,y_n,\omega_n;z)$, define its {\bf length} to be $n$, its {\bf degree} to be its total number of vertices and leaves, and its {\bf weight} to be the triple 
$$|u|:=(d,n,n+m),$$ 
where $d$ is the degree of $u$, $n$ is the length of $u$, and $m$ is the length of $y_1$. 

\begin{prop}
The map $\tildeh:\bfk\NMNE\to  \fnmna$ in Eq.\,\meqref{eq:tildeh} is surjective.
\mlabel{p:fnmnsbmne}
\end{prop}

\begin{proof}
Since $\bfk\NMNE$ is a subspace of $\bfk{\rm MT}_\Omega(X)$ and 
$\fnmna=\bfk{\rm MT}_\Omega(X)/J$, we just need to prove 
$\bfk{\rm MT}_\Omega(X)=\bfk\NMNE+J.$
For this purpose, we just need to prove
\begin{equation}
u \equiv 0 \mod  \bfk \NMNE + J, \quad \forall u\in {\rm MT}_\Omega(X).
\mlabel{e:modsurj2}
\end{equation}
We will verify this by induction on the weight $|u|$ of $u$ (with respect to the lexicographical order). For $|u|=(1,0,0)$ it is immediate since $u$ is already in $\NMNE$.

For the inductive step, suppose that the result holds for $u$ with $(1,0,0)<|u|<(d,n,n+m)$, for $d,n,m\in \Z_+$.
Consider $u=r(y_1,\omega_1,\cdots ,y_n,\omega_n;z)\in  {\rm MT}_\Omega(X)$ with $|u|=(d,n,n+m)$. 

As a preliminary step, we first show that $y_1,\ldots,y_n$  can be arranged in decreasing order of degrees when read from left to right. Explicitly, if the degree of $y_{i+1}$ is greater than the degree of $y_{i}$, for some $i=1,\cdots,n-1$, then we prove that $y_{i+1}$ can be exchanged with $y_{i}$. Applying Eq.~\meqref{e:n21} (see Proposition\,\mref{p:rexpreln}), we have
\begin{equation}
\begin{split}
u=&~r(y_1,\omega_1,\cdots ,y_{i},\omega_{i},y_{i+1},\omega_{i+1},\cdots,y_n,\omega_n;z)\\
\equiv&~r\big(y_1,\omega_1,\cdots y_{i-1},\omega_{i-1}, y_{i}\rhd_{\omega_{i}}y_{i+1},\omega_{i+1},\cdots,y_n,\omega_n;z\big)\\
&-r\big(y_1,\omega_1,\cdots y_{i-1},\omega_{i-1}, y_{i+1}\rhd_{\omega_{i}}y_{i},\omega_{i+1},\cdots,y_n,\omega_n;z\big)\\
&+r(y_1,\omega_1,\cdots ,y_{i-1},\omega_{i-1},y_{i+1},\omega_{i},y_{i},\omega_{i+1},\cdots,y_n,\omega_n;z) ~~\mod{J}\\
\equiv&~r(y_1,\omega_1,\cdots ,y_{i-1},\omega_{i-1},y_{i+1},\omega_{i},y_{i},\omega_{i+1},\cdots,y_n,\omega_n;z) ~~\mod  \bfk \NMNE + J.\\
\end{split}\label{eq:decre}
\end{equation}
Here, the weights of the first two monomials on the right-hand-side are given by 
$$\big|r\big(y_1,\omega_1,\cdots y_{i-1},\omega_{i-1}, y_{i}\rhd_{\omega_{i}}y_{i+1},\omega_{i+1},\cdots,y_n,\omega_n;z\big)\big|=(d,n-1,n-1+m)<|w|,$$
$$\big|r\big(y_1,\omega_1,\cdots y_{i-1},\omega_{i-1}, y_{i+1}\rhd_{\omega_{i}}y_{i},\omega_{i+1},\cdots,y_n,\omega_n;z\big)\big|=(d,n-1,n-1+m)<|w|.$$
So the induction hypothesis applies to these monomials.		
Repeating this process if necessary, we may assume that the degrees of $y_1,\ldots,y_n$ are in the decreasing order reading from left to right.	

We next take three steps in following the three stages of Definition~\mref{d:re} to show that the element $u$ can be reduced to a noncommuting multi-Novikov element. 
Specifically, the first step fulfills the condition in Definition\mref{d:re}.\mref{i:re1}, which makes $u$ into a nest.
The second step fulfills the conditions in Definition\,\mref{d:re}.\mref{i:re2}, which makes $u$ into an ordered nest. 
The third step fulfills the condition in Definition\,\mref{d:re}.\mref{i:re3}, reducing $u$ to a noncommuting multi-Novikov element. 

\noindent {\bf Step 1 (making $u$ a nest).}
If $u$ is a nest, then we move to Step 2. If $u$ is not a nest, by the induction hypothesis, each $y_i$ is a nest for $1\leq i\leq n$. 
Thus the elements $x_2,\cdots, x_m$ in $y_1=r(x_1,\tau_1,\cdots,x_m,\tau_m;z')$ belong to $X$.
Applying Eq.~\eqref{e:n22}, we have
\begin{equation*}
\begin{split}
u\equiv&~ r\big(r(x_1,\tau_1,\cdots,x_m,\tau_m;z'),\omega_1,y_2,\cdots ,y_n,\omega_n;z\big)\\
\equiv&~	r\big(r(x_1,\omega_1,y_2,\cdots ,y_n,\omega_n;z),\tau_1,\cdots,x_m,\tau_m;z'\big)	\mod  J.
\end{split}
\end{equation*}
Since the degree of $r(x_1,\omega_1,y_2,\cdots ,y_n,\omega_n;z)$ is less than $u$,  by the induction hypothesis again,
$$r(x_1,\omega_1,y_2,\cdots ,y_n,\omega_n;z)\equiv v \mod  J$$ for some nest $v$.
Hence
$u\equiv r\big(r(x_1,\omega_1,y_2,\cdots ,y_n,\omega_n;z),\tau_1,\cdots,x_m,\tau_m;z'\big)\equiv 
r\big(v,\tau_1,\cdots,x_m,\tau_m;z'\big), \mod J$. 
Then we proceed to step 2 for the nest $r\big(v,\tau_1,\cdots,x_m,\tau_m;z'\big)$.

\noindent {\bf Step 2 (making $u$ ordered). } 
If $u$ is ordered, we proceed directly to Step 3. So suppose that $u$ is not ordered.
By Remark~\ref{rm:notorder}, $u$ satisfies a condition in one of the three dashed boxes in Figure~\ref{fg:notorder}.
We further combine the three conditions into two cases: Case 1 corresponds to the first dashed box, and Case 2 corresponds to  the second and third dashed boxes,  for $n>0, y_1=r(x_1,\tau_1,\cdots,x_m,\tau_m;z')$.

\noindent{\bf Case 1.} If $m<n$, applying Eq.\,\meqref{e:n22} (thanks to Proposition\,\mref{p:rexpreln}) yields the following congruence of nests
\begin{equation*}
\begin{split}
&r\big(r(x_1,\tau_1,\cdots,x_m,\tau_m;z'),\omega_1,y_2,\cdots ,y_n,\omega_n;z\big)\\
\equiv&~	r\big(r(x_1,\omega_1,y_2,\cdots ,y_n,\omega_n;z),\tau_1,\cdots,x_m,\tau_m;z'\big)	\mod  J.
\end{split}
\end{equation*}
Then the length $n$ of $r(x_1,\omega_1,y_2,\cdots ,y_n,\omega_n;z)$ is strictly greater than the length $m$ of
$$r\big(r(x_1,\omega_1,y_2,\cdots ,y_n,\omega_n;z),\tau_1,\cdots,x_m,\tau_m;z'\big).$$ 
Hence the latter is ordered.

\noindent{\bf Case 2.} If either $m=n$ and $z'<z$, or $z'=z$ and $(\tau_{1},\cdots,\tau_{m})<(\omega_1,\cdots,\omega_n)$,  
then applying Eq.~\eqref{e:n22} gives the following congruence of nests
\begin{equation*}
\begin{split}
&r\big(r(x_1,\tau_1,\cdots,x_m,\tau_m;z'),\omega_1,y_2,\cdots ,y_n,\omega_n;z\big)\\
\equiv	&~r\big(r(x_1,\omega_1,y_2,\cdots ,y_n,\omega_n;z),\tau_1,\cdots,x_m,\tau_m;z'\big)	\mod   J.
\end{split}
\end{equation*}
The right-hand side is ordered.		

We then proceed to Step 3 with the resulting ordered nest obtained in both cases.

\noindent{\bf Step 3~(adjust the leaf decorations such that they are in the increasing order reading from right to left).}
Let  $u=r(y_1,\omega_1,\cdots ,y_n,\omega_n;z)$ be an ordered nest for $|u|=(d,n,n+m)$ and $y_1=r(x_1,\tau_1,\cdots,x_m,\tau_m;z')$.
If $u$ is a noncommuting multi-Novikov element, then $$u \equiv 0 \mod  \bfk \NMNE + J.$$ 
Now suppose that $u$ is not a noncommuting multi-Novikov element. 
By Eq.~\eqref{eq:decre}, we may interchange the decorations of the leaves attached to the same root. 
Without loss of generality, consider the ordered nest
$$r\big(r(x_1,\tau_1,\cdots,x_m,\tau_m;z'),\omega_1,y_2,\cdots ,y_n,\omega_n;z\big),$$ 
where $ x_2\geq\cdots\geq x_m\in X$ and $y_2\geq y_3\geq\cdots\geq y_n\in X$ .

Since the decorations of the leaves attached to the same root commute modulo $J$, it suffices to show that the decorations of the leaves associated with distinct roots are also commutative.
We now show that $x_m$ and $y_2$ can be interchanged. To this end, we proceed as follows. By repeatedly applying Eqs.~\eqref{e:n22} and~\eqref{eq:decre}, we obtain
\begin{equation*}
\begin{split}
&r\big(r(x_1,\tau_1,\cdots,x_m,\tau_m;z'),\omega_1,y_2,\cdots ,y_n,\omega_n;z\big)\\
\overset{\eqref{eq:decre}}{\equiv}&~	r\big(r(x_m,\tau_1,x_1,\cdots,x_{m-1},\tau_m;z'),\omega_1,y_2,\cdots ,y_n,\omega_n;z\big)	\mod  \bfk \NMNE + J\\
\overset{\eqref{e:n22}}{\equiv}&~r\big(r(x_m,\omega_1,y_2,\cdots ,y_n,\omega_n;z),\tau_1,x_1,\cdots,x_{m-1},\tau_m;z'\big)	\mod  \bfk \NMNE + J\\
\overset{\eqref{eq:decre}}{\equiv}&~r\big(r(y_2,\omega_1,x_m,\cdots ,y_n,\omega_n;z),\tau_1,x_1,\cdots,x_{m-1},\tau_m;z'\big)	\mod  \bfk \NMNE + J\\
\overset{\eqref{e:n22}}{\equiv}&~r\big(r(y_2,\tau_1,x_1,\cdots,x_{m-1},\tau_m;z'),\omega_1,x_m,\cdots ,y_n,\omega_n;z\big)	\mod  \bfk \NMNE + J\\
\overset{\eqref{eq:decre}}{\equiv}&~r\big(r(x_1,\tau_1,\cdots,x_{m-1},\tau_{m-1},y_2,\tau_m;z'),\omega_1,x_m,\omega_2,y_3\cdots ,y_n,\omega_n;z\big)	\mod  \bfk \NMNE + J.\\
\end{split}
\end{equation*}
Thus $u \equiv 0 \mod  \bfk \NMNE + J.$

Therefore, for any $u\in {\rm MT}_\Omega(X)$, by performing Steps 1–3, we obtain
\begin{equation*}
u \equiv 0 \mod  \bfk \NMNE + J,
\end{equation*}
as desired.
\end{proof}

\subsubsection{Composition of $\bar{f}$ and $\tildeh$} 
For a differential monomial $u\in \fnmdc$, defined to be a product of generators  $(\partial_{\omega_1}\cdots \partial_{\omega_n})(x)$, define
$$p(u):=\text{number of times that partial derivatives appear in } u - \text{ that of variables}.$$
The following notion is the noncommuting analog of the one given in\,\cite{BrD}. 

\begin{defi}
A differential monomial in $\fnmdc$ is said to satisfy the {\bf populated condition} if $p(u)=-1$.
\mlabel{d:pop}
\end{defi}

\begin{lem} The noncommuting multi-Novikov subalgebra $\fnmdca$ is linearly spanned by all differential monomials in $\fnmdc$ that satisfy the populated condition.
\mlabel{t:fnovgen}
\end{lem}
\begin{proof}
It is obvious that the populated condition holds for each generator $x\in X$ and hence $X\subset \fnmdca$. To prove that $\fnmdca$ is a noncommuting multi-Novikov subalgebra of $\fnmdc$, we just need to check that, for any two differential monomials $u,v$ for which the populated condition holds, the product $u\ast _\omega v=u\partial_\omega(v),\omega\in \Omega$, is a sum of differential monomials for which the populated condition holds.
This is because from $p(u)=p(v)=-1$ we obtain
$$ p(u\ast_\omega v)=p(u\partial_\omega(v))=p(u)+p(\partial_\omega(v))
=p(u)+p(v)+1=-1.$$		
Now we prove that this subalgebra is generated by $X$ by showing that all differential monomials satisfying the populated condition are generated by $X$. We argue by induction on the total number $n$ of partial derivatives appearing in such a differential monomial $u$.

If $n=0$ for $u$, then $u$ is already in $X$.

For the inductive step, let $k \geq1$. 
Assume that all $u$ satisfying the populated condition with $n=k$ are generated by $X$.
Consider a differential monomial $u$ with $n=k+1$ derivations. Singling out a highest order differential variable as the right-most factor, one can write
$$u=u_1(\partial_{\omega_1}\cdots \partial_{\omega_d})(x),x\in X, d\ge 1,$$
for another differential monomial $u_1$. Then there are two cases to consider:

\noindent	{\bf Case 1.} If $d=1$, then $u=u_1\partial_\omega (x)=u_1\ast_\omega x.$ Since the total number  of all partial derivatives of $u_1$ is $k$ and
$p(u_1)=p(u)-p(\partial_\omega (x))=-1$, by the induction hypothesis, $u_1$ can be generated by $X$. Thus $u$ can be generated by $X$.

\noindent		{\bf Case 2.} If $d \geq2$, then since the total number of partial derivatives in $u$ is one less than the total number of all variables, $u$ can be written as 
$$u=u_2x_{i_1}x_{i_2}\cdots x_{i_{d-1}}(\partial_{\omega_1}\cdots \partial_{\omega_d})(x),$$
where $x_{i_j}\in X,j=1,\cdots, d-1$ and $u_2$ satisfies the populated condition. Then, we have
\begin{align*}
& u_2\ast_{\omega_1}\Big(x_{i_1}x_{i_2}\cdots x_{i_{d-1}}(\partial_{\omega_2}\cdots \partial_{\omega_d})(x)\Big)\\
=&~u_2\partial_{\omega_1}\Big(x_{i_1}x_{i_2}\cdots x_{i_{d-1}}(\partial_{\omega_2}\cdots \partial_{\omega_d})(x)\Big)\\
=&~u_2x_{i_1}x_{i_2}\cdots x_{i_{d-1}}(\partial_{\omega_1}\cdots \partial_{\omega_d})(x)
+\sum_{j=1}^{d-1}\Big(u_2x_{i_1}\cdots \widehat{x_{i_j}}\cdots x_{i_{d-1}}(\partial_{\omega_2}\cdots \partial_{\omega_d})(x)\Big)\ast_{\omega_1}x_{i_j}   
\end{align*}
and hence
$$u=~u_2\ast_{\omega_1}\Big(x_{i_1}x_{i_2}\cdots x_{i_{d-1}}(\partial_{\omega_2}\cdots \partial_{\omega_d})(x)\Big)
-\sum_{j=1}^{d-1}\Big(u_2x_{i_1}\cdots \widehat{x_{i_j}}\cdots x_{i_{d-1}}(\partial_{\omega_2}\cdots \partial_{\omega_d})(x)\Big)\ast_{\omega_1}x_{i_j}.
$$
Note that each of $u_2$, $ x_{i_1}x_{i_2}\cdots x_{i_{d-1}}(\partial_{\omega_2}\cdots \partial_{\omega_d})(x)$ and $u_2x_{i_1}\cdots \widehat{x_{i_j}}\cdots x_{i_{d-1}}(\partial_{\omega_2}\cdots \partial_{\omega_d})(x)$, $j=1,\cdots,$ $d-1$ satisfies the populated condition and has $k$ or fewer derivations. Hence by the induction hypothesis, they can be generated by $X$. Therefore, $u$ can also be generated by $X$. This completes the inductive proof.
\end{proof}

We define a linear map by a recursion on the number of edges as follows.
\begin{align}
\mlabel{eq:mtop}
\begin{split}
g:\bfk\NMNE \ra&~ \fnmdca\\
r(z)\mapsto&~ {z},\quad\text{ for } z\in X,\\ 	
r(y_1,\omega_1,\cdots,y_n,\omega_n;z)\mapsto&~ g(y_1) \ast_{\omega_1} g\big(r(y_2,\omega_2,\cdots,y_n,\omega_n;z))\\
&=g(y_1) \partial_{\omega_1} g\big(r(y_2,\omega_2,\cdots,y_n,\omega_n;z)),
\end{split}
\end{align}
for  $z,y_i \in X,i=2,\cdots,n,\omega_i\in \Omega,$ 
where $\ast_{\omega_1}$ is the noncommuting multi-Novikov product on $\fnmdca$. Here the fact that the image of $g$ is in $\fnmdca$ can be checked recursively using Lemma~\mref{t:fnovgen} since from $p(g(y_1))=-1$ and $p\big(g\big(r(y_2,\omega_2,\cdots,y_n,\omega_n;z)\big)\big)=-1$, we obtain
$$p\Big(g(y_1) \partial_{\omega_1} g\big(r(y_2,\omega_2,\cdots,y_n,\omega_n;z)\big)\Big)=-1+1+(-1)=-1.$$

\begin{prop}
\mlabel{p:fnnov}
With the map $\tildeh$ defined in Eq.\,\meqref{eq:tildeh}, we have $\bar{f}\tildeh=g$, that is, the following diagram commutes. 
$$\xymatrix{
\fnmna \ar[d]^{\bar{f}}  && \bfk\NMNE       \ar[ll]_{\tildeh}     \\
\fnmdca       \ar@{<-}[urr]_g &&           }$$
\end{prop}

\begin{proof}
We prove $\bar{f}\tildeh=g$ by induction on the degree $k\geq1$ of noncommuting multi-Novikov elements.
For $k=1$ and hence $x\in \Omegax$, we have
$$\bar{f}\tildeh(x)=\bar{f}(x)=f(x)=x=g(x).$$
Suppose that the diagram commutes for all elements with degree $\leq k$.
For $r(y_1,\omega_1,\cdots,y_n,\omega_n;z)$ in $\bfk\NMNE $ with  degree $k+1$, where $y_1\in\bfk\NMNE $ and $y_i\in X,i=2,\cdots,n$, by the induction hypothesis, we have
\begin{align*}
&\bar{f}\tildeh\Big(r(y_1,\omega_1,\cdots,y_n,\omega_n;z)\Big)\\
=&~\bar{f}\Big(\tildeh(y_1) \rhd_{\omega_1} \tildeh(r(y_2,\omega_2,\cdots,y_n,\omega_n;z))\\
=&~\bar{f}\tildeh(y_1) \ast_{\omega_1} \bar{f}\tildeh(r(y_2,\omega_2,\cdots,y_n,\omega_n;z))\\
=&~g(y_1) \ast_{\omega_1} g(r(y_2,\omega_2,\cdots,y_n,\omega_n;z))\\
=&~g\Big(r(y_1,\omega_1,\cdots,y_n,\omega_n;z)\Big),
\end{align*}
as required.
\end{proof}

\subsubsection{An auxiliary bijection}

We define another linear map $$\phi:\bfk\NMNE \ra \fnmdca$$ by a recursion on the total number of times that elements in $X\sqcup\Omega$ appear in $r\in \NMNE$: 
\begin{align}\mlabel{e:phi}
\begin{split}
&\phi(r(z))=~ z,\quad\text{ for } z\in X, \\
&\phi\big(r(y_1 ,{\omega_1},y_2, \cdots,y_n,\omega_n;z)\big)=
\phi(y_1)y_2\cdots y_n\partial_{\omega_1}\cdots\partial_{\omega_n}(z). 
\end{split}
\end{align}
Here, $p(\phi(y_1))=-1$ due to the induction hypothesis and Lemma\,\mref{t:fnovgen}. 
So 
$$p(\phi(y_1)y_2\cdots y_n\partial_{\omega_1}\cdots\partial_{\omega_n}(z))=p(\phi(y_1))+p(y_2\cdots y_n\partial_{\omega_1}\cdots\partial_{\omega_n}(z))=-1+0=-1.$$ Hence by Lemma\,\mref{t:fnovgen} again, $\phi(r(y_1 ,{\omega_1},y_2, \cdots,y_n,\omega_n;z))$ is in $\fnmdca$. 

\begin{ex}
Assume $a<b<c<d$ in $X$. Then the noncommuting multi-Novikov element $$r(r(r(d,\beta,d,\alpha,c,\gamma,c,\beta;c),\beta,c,\beta,b,\alpha;d),\gamma,a,\beta;a)$$ in ~Example \mref{g:nmne} corresponds  to {\small
$$\phi(r(r(r(d,\beta,d,\alpha,c,\gamma,c,\beta;c),\beta,c,\beta,b,\alpha;d),\gamma,a,\beta;a))=d^2c^3ba(\partial_\beta\partial_\alpha\partial_\gamma \partial_\beta(c))(\partial_\beta \partial_\beta \partial_\alpha(d))(\partial_\gamma \partial_\beta(a))$$
}
 in $\fnmdca$.
\end{ex}

Moreover, $\fnmdca$ is an $\N$-graded vector space with respect to the number of occurrences of all derivations; while $\bfk\NMNE$ is also an $\N$-graded space, where the grading is given by the total number of edges. By the construction of $\phi: \bfk\NMNE\to  \fnmdca$, it is an $\N$-graded map.

\begin{prop}
The $\N$-graded linear map $\phi:\bfk\NMNE \ra~ \fnmdca$ is bijective. 
\mlabel{p:nmneipcd}
\end{prop}
\begin{proof}
We prove the bijectivity by constructing a linear map $\psi:  \fnmdca \ra \bfk\NMNE$ which will serve as the inverse map.
The construction is given by induction on the total number $k\geq 0$ of partial derivatives appearing in a basis element in $\fnmdca$. Before constructing the map $\psi$, we first rewrite the element in $ \fnmdca$ as follows, making use of the commutativity of the multiplication on the multi-differential algebra.
\begin{itemize}
\item For $k=0, u=x, x\in X$.
\item For $k\geq 1$, mirroring the conditions in Definition\,\mref{d:re}, $u$ can be written uniquely as 
\begin{align*}
u=&x_{0}x_{1,1}\cdots x_{1,n_1-1}\partial_{\omega_{1,1}}\cdots\partial_{\omega_{{1,n_1}}}(x_{1,n_1})x_{2,1}\cdots x_{2,n_2-1}\partial_{\omega_{2,1}}\cdots\partial_{\omega_{2,n_2}}(x_{2,n_2})\\
&\cdots x_{s,1}\cdots x_{s,n_s-1}\partial_{\omega_{s,1}}\cdots\partial_{\omega_{s,n_s}}(x_{s,n_s}),
\end{align*} 
where the elements are listed according to the following rules.
$$x_{0}\geq x_{1,1}\geq x_{1,2} \geq \cdots \geq x_{1,n_1-1}\geq x_{2,1}\geq x_{2,2} \geq \cdots \geq x_{2,n_2-1}\geq 
\cdots \geq x_{s,1}\geq \cdots \geq x_{s,n_s-1}\in X,$$
and $n_1\geq n_2\geq \cdots\geq n_s$, $\sum_{i=1}^{s} n_i=k$. If $n_j=n_{j+1}$ for some $1\leq j< s-1$, then require either $x_{j,n_j}>x_{j+1,n_{j+1}}$, or $x_{j,n_j}=x_{j+1,n_{j+1}}$ and  $(\omega_{j,1},\cdots,\omega_{j,n_j})\geq (\omega_{j+1,1},\cdots,\omega_{j+1,n_{j+1}})$ under the lexicographical order. 
\end{itemize}
For notational convenience, we denote $u=u_1 x_{s,1}\cdots x_{s,n_s-1}\partial_{\omega_{s,1}}\cdots\partial_{\omega_{s,n_s}}(x_{s,n_s})$, where $$u_1=x_{0}x_{1,1}\cdots x_{1,n_1-1}\partial_{\omega_{1,1}}\cdots\partial_{\omega_{{1,n_1}}}(x_{1,n_1})\cdots x_{s-1,1}\cdots x_{s-1,n_{s-1}-1}\partial_{\omega_{s-1,1}}\cdots\partial_{\omega_{s-1,n_{s-1}}}(x_{s-1,n_{s-1}})$$
is another element in $\fnmdca$. 

Now we construct the map $\psi$ as follows based on the form of the elements in $\fnmdca$ as above.
\begin{align}
\begin{split}
\psi:  \fnmdca \ra&~ \bfk\NMNE,\\
z\mapsto&~ {r(z)},\quad\text{ for } z\in X, \\u_1 x_{s,1}\cdots x_{s,n_s-1}\partial_{\omega_{s,1}}\cdots\partial_{\omega_{s,n_s}}(x_{s,n_s})\mapsto
&~ r(\psi(u_1),\omega_{s,1},x_{s,1},\omega_{s,2},\cdots,x_{s,n_s-1},\omega_{s,n_s};x_{s,n_s}).
\end{split}
\end{align}
For example,  let $x_1x_2\partial_{\omega_1}\partial_{\omega_2}(x_3)\partial_{\omega_3}(x_4)$ with $x_1\geq x_2$ and  $x_3\geq x_4$. Then  
$$\psi(x_1x_2\partial_{\omega_1}\partial_{\omega_2}(x_3)\partial_{\omega_3}(x_4))=r(r(x_1,\omega_1,x_2,\omega_2;x_3),\omega_3;x_4).$$
From the choice of the above form of elements in $\fnmdca$, the image of $\psi$ is in $\bfk\NMNE$. 

We verify $\psi\phi=\id_{\bfk\NMNE}$ by induction on the total number $\ell\geq1$ of times that the elements in $\Omega\sqcup X$ appearing in an element $u$ in $\bfk\NMNE$. 

If $\ell=1$, then $u=r(z), z\in X$, so we have $\psi\phi(r(z))=\psi(z)=r(z)$.

For the inductive step, suppose that $\psi\phi(u)=u$ for all $u\in \NMNE$ with total number $\ell$. Then for $u=r(y_1 ,{\omega_1},y_2, \cdots,y_n,\omega_n;z)\in \NMNE$ with total number $\ell+1$, we have
\begin{align*}
\psi\phi(r(y_1 ,{\omega_1},y_2, \cdots,y_n,\omega_n;z))&=\psi(\phi(y_1)y_2\cdots y_n\partial_{\omega_1}\cdots\partial_{\omega_n}(z))\\&=r(\psi\phi(y_1) ,{\omega_1},y_2, \cdots,y_n,\omega_n;z)\\&=r(y_1 ,{\omega_1},y_2, \cdots,y_n,\omega_n;z).
\end{align*}
One can similarly prove $\phi \psi=\id_{\fnmdca}$.
\end{proof}

\subsection{The proof of Theorem\,\mref{t:fncmncasubnov}}
\mlabel{ss:proof}
With the preparations in Section\,\mref{ss:prep}, we now give the proof of Theorem\,\mref{t:fncmncasubnov}.

\noindent
\mref{i:main1}.
Proposition~\mref{p:fnmnsbmne} states that $\tildeh$ is surjective. It remains to show that $\tildeh$ is injective. Since $\bar{f}\tildeh=g$ by Proposition~\mref{p:fnnov}, we just need to show that $g$ is injective. 

For the $\N$-graded spaces $\bfk\NMNE$ and $\fnmdca$ in Proposition\,\mref{p:nmneipcd}, denote 
$$\bfk\NMNE=\bigoplus_{k\geq 0}\bfk\NMNE_k, \quad \fnmdca=\bigoplus_{k\geq 0}\fnmdca_k,k\in \N,$$
where $\bfk\NMNE_k$ denotes the subspace of $\bfk\NMNE$ spanned by all typed decorated trees with $k$ edges, and
$\fnmdca_k$ denotes the subspace of $\fnmdca$ spanned by all monomials with $k$ derivations.

We show that $g$ is an $\N$-graded linear map by checking 
$$ g(\NMNE_k) \subseteq \fnmdca_k$$
inductively on $k\geq 0$. 
For $k=0$, then $u\in X=\bfk\NMNE_0$, and we have $g(u)=u\in X\subseteq \fnmdca_0$. For the inductive step, consider $r(y_1,\omega_1,\cdots,y_n,\omega_n;z)\in \bfk\NMNE_{k+1}$ for a given $k\geq 0$. Thus $n\geq 1$. 
By the definition of $g$, we have
\begin{align*}
g(r(y_1,\omega_1,\cdots,y_n,\omega_n;z))=&~ g(y_1) \ast_{\omega_1} g\big(r(y_2,\omega_2,\cdots,y_n,\omega_n;z))\\
=&~g(y_1) \partial_{\omega_1} g\big(r(y_2,\omega_2,\cdots,y_n,\omega_n;z)). 
\end{align*}
Notice that the number of edges of the typed decorated trees $y_1$ and $r(y_2,\omega_2,\cdots,y_n,\omega_n;z)$ are $k+1-n$ and $n-1$, respectively. 
Thus by the induction hypothesis, we get
$$g(y_1)\in\fnmdca_{k+1-n} \text{ and }  g(r(y_2,\omega_2,\cdots,y_n,\omega_n;z))\in\fnmdca_{n-1}.$$
Further by the Leibniz rule, the number of derivations in each term of the expansion of\\
$\partial_{\omega_1} g\big(r(y_2,\omega_2,\cdots,y_n,\omega_n;z))$ is $n$. Consequently, the number of partial derivatives in a monomial of $g(y_1) \partial_{\omega_1} g\big(r(y_2,\omega_2,\cdots,y_n,\omega_n;z))$ is $(k+1-n)+n=k+1$, showing $g(r(y_1,\omega_1,\cdots,y_n,\omega_n;z))\in\fnmdca_{k+1}$. Therefore, $g$ is an $\N$-graded linear map. 
Moreover, since both $\fnmna$ and $\fnmdca$ are noncommuting multi-Novikov algebras generated by X, $\bar{f}$ is surjective. By Proposition~\mref{p:fnmnsbmne},  $\tildeh$ is surjective. Therefore, $g=\bar{f}\tildeh$ is also surjective. So the map $g|_{\bfk {\rm MNE}_{\rm NC,\Omega}(X)_n}:\bfk {\rm MNE}_{\rm NC,\Omega}(X)_n\to{\rm PCD}_{\rm NC, \Omega}(X)_n$ is also surjective, for each $n\in \N$.

Now we prove the injectivity of $g$ by showing that $\ker g=0$. 
Let $u\in \ker(g)$ be given. Let $X_0\subset X$ and $\Omega_0\subset \Omega$ be the finite subsets appearing in $u$. Then $u$ is contained in $\bfk {\rm MNE}_{\rm NC,\Omega_0}(X_0)\subset \bfk\NMNE$.
Furthermore, the surjective $\N$-graded linear map $g:\bfk \NMNE\to \fnmna$ restricts to a surjective $\N$-graded linear map 
$$g: \bfk {\rm MNE}_{\rm NC,\Omega_0}(X_0)\to {\rm PCD}_{\rm NC, \Omega_0}(X_0),$$ 
and hence a surjective linear map 
\begin{equation}
g: \bfk {\rm MNE}_{\rm NC,\Omega_0}(X_0)_n\to {\rm PCD}_{\rm NC, \Omega_0}(X_0)_n,n\in \N.
\mlabel{e:phihomo}
\end{equation}

Now note that the $\N$-graded linear bijection $\phi : \bfk\NMNE\to  \fnmdca$ in Proposition\,\mref{p:nmneipcd} restricts to an $\N$-graded linear bijection
$$\phi: \bfk {\rm MNE}_{\rm NC,\Omega_0}(X_0)\to {\rm PCD}_{\rm NC, \Omega_0}(X_0),$$ 
and hence a linear bijection
$$\phi: \bfk {\rm MNE}_{\rm NC,\Omega_0}(X_0)_n\to {\rm PCD}_{\rm NC, \Omega_0}(X_0)_n,n\in \N.$$ 
Since the last two spaces are finite-dimensional, we have
$$ \mathrm{dim}\big(\bfk {\rm MNE}_{\rm NC,\Omega_0}(X_0)_n\big)=\mathrm{dim}({\rm PCD}_{\rm NC, \Omega_0}(X_0)_n), n\in \N.$$ 
This equality of dimensions implies that the surjective linear map 
$$g: \bfk {\rm MNE}_{\rm NC,\Omega_0}(X_0)_n{\to}{\rm PCD}_{\rm NC, \Omega_0}(X_0)_n$$ 
in Eq.\,\meqref{e:phihomo} must be injective for all $n\geq 0$. Hence $g$ is injective on $\bfk {\rm MNE}_{\rm NC,\Omega_0}(X_0)$. Therefore, the given element $u\in \ker g \cap \bfk {\rm MNE}_{\rm NC,\Omega_0}(X_0)$ must be zero, showing that the linear map $g:\bfk \NMNE\to \fnmdca$ is injective, as desired. 

\smallskip

\noindent
\mref{i:main2}. This is a direct consequence of Item\,\mref{i:main1}. 

\smallskip 

\noindent
\mref{i:main3}. 
The proof of Item\,\mref{i:main1} gives the bijectivity of $g$. Thus $\bar{f}=g\tildeh^{-1}$ is also bijective and hence an isomorphism of noncommuting multi-Novikov algebras. 

\noindent
\mref{i:main4}. It follows directly from Item\,\mref{i:main3}, since $\fnmna$ is a free noncommuting multi-Novikov algebra on $X$. 

This completes the proof of Theorem\,\mref{t:fncmncasubnov}. 

\smallskip

\noindent
{\bf Acknowledgments.}
Xiaoyan Wang is  supported by  the National Key R\&D Program of China (2024YFA1013803)  and by Shanghai Key Laboratory of PMMP (22DZ2229014).
Huhu Zhang is supported by the Scientific Research Foundation of High Level Talents of Yulin University (2025GK12), Young Talent
Fund of Association for Science and Technology in Shaanxi, China (20250530) and Young Talent Fund of
Association for Science and Technology in Yulin (20250711).

\smallskip

\noindent
{\bf Declaration of interests. } The authors have no conflicts of interest to disclose.

\smallskip

\noindent
{\bf Data availability.} Data sharing is not applicable as no data were created or analyzed in this study.

\end{document}